\documentclass{article}
\usepackage{amsmath, amssymb, amsthm, xparse, tikz, listings, xcolor, caption}
\usepackage{soul}
\usepackage[colorlinks=true, linkcolor=blue, citecolor=red]{hyperref}

\newtheorem{lemma}{Lemma}[section]
\newtheorem{theorem}[lemma]{Theorem}
\newtheorem{proposition}[lemma]{Proposition}

\newtheorem{question}{Question}
\newtheorem{corollary}[lemma]{Corollary}

\theoremstyle{definition}
\newtheorem{definition}[lemma]{Definition}

\theoremstyle{remark}
\newtheorem{remark}[lemma]{Remark}
\newtheorem{notation}[lemma]{Notation}

\definecolor{TianyiBlue}{RGB}{102,204,255}

\NewDocumentCommand{\drawSequence}{mmmm O{} O{}}{%
    \def\radius{#1}%
    \def\xspacing{#2}%
    \def\yspacing{#3}%
    \def\sequence{#4}%
    \begingroup
      \def\seqReverse{}%
      \newcount\itemCount
      \itemCount=0
      \foreach \val in \sequence {
          \xdef\seqReverse{\val,\seqReverse}%
          \global\advance\itemCount by 1
      }
    \endgroup
    \begingroup
      \def\lastNegative{}%
      \foreach \v in \seqReverse {%
        \ifx\lastNegative\empty
          \ifnum\v<0
            \xdef\lastNegative{\v}%
          \fi
        \fi
      }%
      \newcount\c
      \ifx\lastNegative\empty
        \c=0
      \else
        \pgfmathtruncatemacro{\tempC}{-1*\lastNegative}%
        \global\c=\tempC
      \fi
      \newcount\counter
      \counter=0
      \def\a{}%
      \def\b{}%
      \foreach \v in \seqReverse {%
          \global\advance\counter by 1
          \ifnum\counter=\numexpr\c+3\relax
            \xdef\a{\v}%
          \fi
          \ifnum\counter=\numexpr\c+2\relax
            \xdef\b{\v}%
          \fi
      }%
    \endgroup
    \newcount\countA
    \newcount\countB
    \begingroup
      \ifx\a\empty
        \global\countA=9999
      \else
        \pgfmathtruncatemacro{\valA}{\a}%
        \global\countA=\valA
      \fi
      \ifx\b\empty
        \global\countB=9999
      \else
        \pgfmathtruncatemacro{\valB}{\b}%
        \global\countB=\valB
      \fi
    \endgroup
    \newcount\rowIndex
    \global\rowIndex=1
    \begin{figure}[hbpt]
      \centering
      \begin{tikzpicture}
        \pgfmathsetmacro{\currenty}{0}
        \pgfmathsetmacro{\prevx}{0}
        \global\def\shouldDrawLine{0}
        \foreach \s in \sequence {
          \ifnum\s<0
            \pgfmathsetmacro{\currentx}{\prevx + \xspacing}
            \ifnum\rowIndex<\countA
              \node[color=black, scale=\radius*12] at (\currentx, \currenty) {\the\numexpr-\s};
            \else
              \ifnum\rowIndex<\countB
                \node[color=black, scale=\radius*12] at (\currentx, \currenty) {\the\numexpr-\s};
              \else
                \node[color=black, scale=\radius*12] at (\currentx, \currenty) {\the\numexpr-\s};
              \fi
            \fi
            \pgfmathsetmacro{\currenty}{\currenty - \yspacing}
            \global\let\currenty\currenty
            \global\def\shouldDrawLine{0}
            \global\advance\rowIndex by1
          \else
            \pgfmathsetmacro{\currentx}{\s*\xspacing}
            \coordinate (current) at (\currentx, \currenty);
            \ifnum\rowIndex<\countA
              \draw[color=black] (current) circle [radius=\radius];
            \else
              \ifnum\rowIndex<\countB
                \draw[color=black] (current) circle [radius=\radius];
              \else
                \draw[color=black] (current) circle [radius=\radius];
              \fi
            \fi
            \ifnum\shouldDrawLine=1
              \ifnum\rowIndex<\countA
                \draw[color=black] ({\prevx + \radius}, \currenty) -- ({\currentx - \radius}, \currenty);
              \else
                \ifnum\rowIndex<\countB
                  \draw[color=black] ({\prevx + \radius}, \currenty) -- ({\currentx - \radius}, \currenty);
                \else
                  \draw[color=black] ({\prevx + \radius}, \currenty) -- ({\currentx - \radius}, \currenty);
                \fi
              \fi
            \fi
            \global\let\prevx\currentx
            \global\def\shouldDrawLine{1}
          \fi
        }
      \end{tikzpicture}
      \caption{#5}
      \IfValueT{#6}{\label{#6}}
    \end{figure}
}

\NewDocumentCommand{\drawColoredSequence}{mmmm O{} O{}}{%
    \def\radius{#1}%
    \def\xspacing{#2}%
    \def\yspacing{#3}%
    \def\sequence{#4}%
    \begingroup
      \def\seqReverse{}%
      \newcount\itemCount
      \itemCount=0
      \foreach \val in \sequence {
          \xdef\seqReverse{\val,\seqReverse}%
          \global\advance\itemCount by 1
      }
    \endgroup
    \begingroup
      \def\lastNegative{}%
      \foreach \v in \seqReverse {%
        \ifx\lastNegative\empty
          \ifnum\v<0
            \xdef\lastNegative{\v}%
          \fi
        \fi
      }%
      \newcount\c
      \ifx\lastNegative\empty
        \c=0
      \else
        \pgfmathtruncatemacro{\tempC}{-1*\lastNegative}%
        \global\c=\tempC
      \fi
      \newcount\counter
      \counter=0
      \def\a{}%
      \def\b{}%
      \foreach \v in \seqReverse {%
          \global\advance\counter by 1
          \ifnum\counter=\numexpr\c+3\relax
            \xdef\a{\v}%
          \fi
          \ifnum\counter=\numexpr\c+2\relax
            \xdef\b{\v}%
          \fi
      }%
    \endgroup
    \newcount\countA
    \newcount\countB
    \begingroup
      \ifx\a\empty
        \global\countA=9999
      \else
        \pgfmathtruncatemacro{\valA}{\a}%
        \global\countA=\valA
      \fi
      \ifx\b\empty
        \global\countB=9999
      \else
        \pgfmathtruncatemacro{\valB}{\b}%
        \global\countB=\valB
      \fi
    \endgroup
    \newcount\rowIndex
    \global\rowIndex=1
    \begin{figure}[htpb]
      \centering
      \begin{tikzpicture}
        \pgfmathsetmacro{\currenty}{0}
        \pgfmathsetmacro{\prevx}{0}
        \global\def\shouldDrawLine{0}
        \foreach \s in \sequence {
          \ifnum\s<0
            \pgfmathsetmacro{\currentx}{\prevx + \xspacing}
            \ifnum\rowIndex<\countA
              \node[color=black, scale=\radius*12] at (\currentx, \currenty) {\the\numexpr-\s};
            \else
              \ifnum\rowIndex<\countB
                \node[color=TianyiBlue, scale=\radius*12] at (\currentx, \currenty) {\the\numexpr-\s};
              \else
                \node[color=black, scale=\radius*12] at (\currentx, \currenty) {\the\numexpr-\s};
              \fi
            \fi
            \pgfmathsetmacro{\currenty}{\currenty - \yspacing}
            \global\let\currenty\currenty
            \global\def\shouldDrawLine{0}
            \global\advance\rowIndex by1
          \else
            \pgfmathsetmacro{\currentx}{\s*\xspacing}
            \coordinate (current) at (\currentx, \currenty);
            \ifnum\rowIndex<\countA
              \draw[color=black] (current) circle [radius=\radius];
            \else
              \ifnum\rowIndex<\countB
                \draw[color=TianyiBlue] (current) circle [radius=\radius];
              \else
                \draw[color=black] (current) circle [radius=\radius];
              \fi
            \fi
            \ifnum\shouldDrawLine=1
              \ifnum\rowIndex<\countA
                \draw[color=black] ({\prevx + \radius}, \currenty) -- ({\currentx - \radius}, \currenty);
              \else
                \ifnum\rowIndex<\countB
                  \draw[color=TianyiBlue] ({\prevx + \radius}, \currenty) -- ({\currentx - \radius}, \currenty);
                \else
                  \draw[color=black] ({\prevx + \radius}, \currenty) -- ({\currentx - \radius}, \currenty);
                \fi
              \fi
            \fi
            \global\let\prevx\currentx
            \global\def\shouldDrawLine{1}
          \fi
        }
      \end{tikzpicture}
      \caption{#5}
      \IfValueT{#6}{\label{#6}}
    \end{figure}
}

\NewDocumentCommand{\drawSimplePattern}{mmmm O{} O{}}{%
    \def\radius{#1}%
    \def\xspacing{#2}%
    \def\yspacing{#3}%
    \def\sequence{#4}%
    \begin{figure}[htpb]
      \centering
      \begin{tikzpicture}
        \pgfmathsetmacro{\currenty}{0}%
        \global\def\shouldDrawLine{0}%
        \foreach \s in \sequence {
          \ifnum\s<0
            \pgfmathsetmacro{\currenty}{\currenty - \yspacing}%
            \global\let\currenty\currenty
            \global\def\shouldDrawLine{0}%
          \else
            \pgfmathsetmacro{\currentx}{\s*\xspacing}%
            \coordinate (currentcoord) at (\currentx, \currenty);
            \draw (currentcoord) circle [radius=\radius];
            \ifnum\shouldDrawLine=1
              \draw ({\prevx + \radius}, \currenty) -- ({\currentx - \radius}, \currenty);
            \fi
            \global\let\prevx\currentx
            \global\def\shouldDrawLine{1}%
          \fi
        }
      \end{tikzpicture}%
      \caption{#5}
      \IfValueT{#6}{\label{#6}}
    \end{figure}
}

\NewDocumentCommand{\drawSimplePatternFigure}{mmmm}{%
    \def\radius{#1}%
    \def\xspacing{#2}%
    \def\yspacing{#3}%
    \def\sequence{#4}%
    \begin{tikzpicture}
      \pgfmathsetmacro{\currenty}{0}%
      \global\def\shouldDrawLine{0}%
      \foreach \s in \sequence {
        \ifnum\s<0
          \pgfmathsetmacro{\currenty}{\currenty - \yspacing}%
          \global\let\currenty\currenty
          \global\def\shouldDrawLine{0}%
        \else
          \pgfmathsetmacro{\currentx}{\s*\xspacing}%
          \coordinate (currentcoord) at (\currentx, \currenty);
          \draw (currentcoord) circle [radius=\radius];
          \ifnum\shouldDrawLine=1
            \draw ({\prevx + \radius}, \currenty) -- ({\currentx - \radius}, \currenty);
          \fi
          \global\let\prevx\currentx
          \global\def\shouldDrawLine{1}%
        \fi
      }
    \end{tikzpicture}%
}

\DeclareMathOperator{\crit}{crit}
\DeclareMathOperator{\Iter}{Iter}

\begin{document}

\title{Notes on Laver Tables}
\author{Renrui Qi}
\date{July 04, 2026}
\maketitle
\setlength{\parindent}{10pt}

\begin{abstract}
    We present some new lower bound estimates for certain numbers in Laver table theory and introduce several related structures of interest.
\end{abstract}

\section{Introduction}

The exploration of large cardinal axioms, particularly exceptionally strong ones such as rank-into-rank axioms, has frequently led to the discovery of unexpected and intricate mathematical structures. In the course of investigating these potent axioms, Richard Laver encountered a fascinating class of objects arising from the axiom I3. These objects, now known as Laver tables, possess a rich combinatorial nature.

A profound consequence of Laver's work was demonstrating that the axiom I3 implies the totality of a specific computable function, which, following Dougherty's notation, we denote as $F(n)$. To date, no axiom system weaker than ZFC+I3 has been shown to be strong enough to prove the totality of $F$. The sheer strength of the axiom system required to prove its totality indicates that $F$ is an extremely fast-growing function. In 1992, Dougherty~\cite{Dougherty} established a lower bound for $F(n)$ that marginally surpasses the growth rate of the Ackermann function. 

Although Dougherty suggested that this lower bound should be improved, no such advancements have been made in the three decades since. This paper significantly improves upon Dougherty's lower bound for $F(n)$, demonstrating that $F$ grows so rapidly that its totality cannot be proven within Peano Arithmetic.

I thank the AI for math team of Peking University, Jiang Jiedong, Zhang Zhiyuan and Chen Leheng so much, for using Lean 4 to formalize the main theorem of this work, and finding many typos during their formalization.

\medskip

We recall some notation and facts from~\cite{Dougherty}.  

Recall that I3 states that there is a cardinal~$\lambda$ with a non-trivial elementary embedding \(j\colon V_\lambda\to V_\lambda\). We fix such a cardinal~\(\lambda\) and an embedding~\(j\). 

Let \(\kappa_0 = \crit(j)\), the critical point of~$j$ (the least ordinal~$\kappa$ with $j(\kappa)>\kappa$), and for \(n \in \omega\), set \(\kappa_{n+1} = j(\kappa_n)\). It is known that \(\lim_{n \to \omega} \kappa_{n} = \lambda\).

For two elementary embeddings \( l, l' : V_\lambda \to V_\lambda\), define
\[
l(l') = \bigcup_{\theta \in \lambda} \Bigl(l(l' \cap V_\theta)\Bigr).
\]

$l(l')$ is also an elementary embedding from $V_\lambda$ to itself, the result of \emph{applying} $l$ to~$l'$.

Let \(\Iter(j)\) be the set of elementary embeddings \(V_\lambda \to V_\lambda\) generated by \(j\) through composition and the application operation defined above. Let \(\crit^*(j)\) be the set of critical points of elements in \(\mathrm{Iter}(j)\). This set has order type \(\omega\), and we write \(\mathrm{crit}^*(j) = \{\gamma_0 < \gamma_1 < \cdots\}\). Let \(F(n)\) be the natural number such that \(\gamma_{F(n)} = \kappa_n\). Under the axiom I3, \(F\) is computable and independent of the particular choice of \(\lambda\) and \(j\).

From~\cite{Dougherty}, we know that $F(0)=0, F(1)=1, F(2)=2, F(3)=4, F(4)>\mathrm{Ack}\bigl(9, \mathrm{Ack}(8, \mathrm{Ack}(8, 254))\bigr)$, and $F(n)$ dominates $\mathrm{Ack}(n,n)$ (here $\mathrm{Ack}$ denotes the Ackerman function).

Our main theorem is:

\begin{theorem} \label{thm:peano_arithmetic}
  $F(n)$ dominates every computable function whose totality is provable in Peano Arithmetic.
\end{theorem}

Also, we will give a much higher lower bound for $F(4)$.
Define \( j_{(1)} = j \), \( j_{(n+1)} = j_{(n)}(j)\), and \(\kappa^m_n = j_{(m)}(\kappa_n)\).
Dougherty showed that 
\[
  \kappa_3 < \kappa_3^{11} < \kappa_3^{10} < \kappa_3^{9} < \kappa_4.
\]
In~\cite{Dougherty} and~\cite{Dougherty2}, Dougherty proved the following.

\begin{proposition} \ 
\begin{enumerate}
    \item \(\kappa^{11}_3 \ge \gamma_{2^{2^{2059}}}\). (Dougherty also conjectured that this is in fact an equality.)
    \item \(\kappa^9_3 > \gamma_{\mathrm{Ack}(4, \mathrm{Ack}(4, 254))}\).
\end{enumerate}
\end{proposition}

In Section~\ref{sec:lower_bounds}, we strengthen these results by proving the following.

\begin{proposition} \label{numberbound}
\ 
\begin{enumerate}
    \item \(\kappa^{11}_3 \ge \gamma_{m_{\omega+2}(2)}\).
    \item \(\kappa^{10}_3 \ge \gamma_{N}\), where \(N = 2^{f(p_{init},1)}\).
\end{enumerate}
\end{proposition}

Here, \(m_\alpha(n)\) is the Steinhaus--Moser function: 
\[
m_0(n) = n^n, 
\quad 
m_{\alpha+1}(n) = m_\alpha^n(n) \text{ for each } \alpha \le \omega+1, 
\quad 
m_\omega(n) = m_n(n),
\]
where the superscript on a function indicates iteration: \(f^1(x) = f(x)\), and \(f^{n+1}(x) = f\bigl(f^n(x)\bigr)\). It is straightforward to see that \(m_{\omega+2}(2)\) exceeds the famous Graham’s number. On the other hand, \(f(p_{init},1)\) is a very large number that we will define later.

We prove these lower bounds in Section~\ref{sec:lower_bounds}. In Section~\ref{sec:laver_table_yarns}, we introduce a new structure called the \emph{Laver table yarn} (LTY). We then present some applications of LTY, including the existence of a computable function that is known to exist only under the axiom I2.

Finally, in Section~\ref{sec:future_work}, we propose some open questions.

\section{Basic Laver patterns}
\label{sec:basic_laver_patterns}

Similar to Dougherty's proof, our improvements to the lower bound are also based on the combinatorial study of the effect of the application operation on the critical points.

Before developing our machinery, we recall the proof of a result in~\cite{Dougherty} to illustrate the flavor of our proofs.

For an elementary embedding \( k : V_\lambda \to V_\lambda \), we write
\[
k : \theta_0 \mapsto \theta_1 \mapsto \theta_2 \mapsto \dots \mapsto \theta_n
\]
if \(\theta_0 = \mathrm{crit}(k)\) and, for \(n' \le n-1\), we have \(\theta_{n'+1} = k(\theta_{n'})\).

The following lemma is straightforward and will be frequently used.

\begin{lemma}\label{lem:l_k_application}
  If
\[
k : \theta_0 \mapsto \theta_1 \mapsto \dots \mapsto \theta_n,
\]
then for another embedding \(l\),
\[
l(k) : \, l(\theta_0) \mapsto l(\theta_1) \mapsto \dots \mapsto l(\theta_n).
\]
\end{lemma}

\begin{definition} \label{def:m-lined}
  Let $m$ be a positive integer, and let $\theta_1 < \theta_2 < \theta_3 < \lambda$. We say that $\theta_1, \theta_2, \theta_3$ are \textit{$m$-lined} if there exist $\theta_2 = \theta'_0 < \cdots < \theta'_{m} = \theta_3$ such that for each $0 < i \le m$, there is $j_i \in \mathrm{Iter}(j)$ such that $j_i: \theta_1 \mapsto \theta'_{i-1} \mapsto \theta'_{i}$.
\end{definition}

Note that if $j_1,...,j_{m}$ witness that $\theta_1,\theta_2,\theta_3$ are $m$-lined, then $j_1$, $j_2$,\dots, $j_{m-2}$, $j_m(j_m)(j_{m-1})$ witness that $\theta_1,\theta_2,\theta_3$ are $(m-1)$-lined. This means that larger $m$ means a stronger condition.

\begin{lemma}\label{lem:lined_composition}
  If $\theta_1,\theta_2,\theta_3$ are $m$-lined, and $\theta_2,\theta_3,\theta_4$ are $m'$-lined, then $\theta_1,\theta_2,\theta_4$ are $m\cdot 2^{m'}$-lined.
\end{lemma}

\begin{proof}
  We use induction on $m'$. Let the cardinals 
  \[
     \theta_2 = \mu_0 < \mu_1 < \cdots < \mu_m = \theta_3
   \]
   and the embeddings $l_i : \theta_1\mapsto \mu_{i-1}\mapsto \mu_i$ witness that $\theta_1,\theta_2,\theta_3$ are $m$-lined. 

   \smallskip

  The base case is $m'=1$. Let $k$ be an embedding with $k:\theta_2\mapsto \theta_3\mapsto \theta_4$. Then for $i=1,\dots, m$, 
   \[
   k(l_i): \theta_1\mapsto k(\mu_{i-1})\mapsto k(\mu_i)
   \]
   and $\theta_3 = k(\mu_0) < k(\mu_1) < \cdots < k(\mu_m) = \theta_4$, so $l_1$, \dots, $l_m$, $k(l_1)$, \dots, $k(l_m)$ show that $\theta_1,\theta_2,\theta_4$ are $2m$-lined. 

  \smallskip

  Suppose that $m'>1$. Let the cardinals $\theta_3  = \mu'_0 < \mu'_1 < \cdots < \mu'_{m'} = \theta_4$ and the embeddings $l'_1$, $l'_2$, \dots, $l'_{m'-1}$ witness that $\theta_2,\theta_3,\theta_4$ are $m'-1$-lined. 

  By the induction hypothesis, $\theta_1,\theta_2,\mu'_{m'-1}$ are $m\cdot 2^{m'-1}$-lined. 

  Further:
  \begin{itemize}
    \item The embeddings $l'_{m'}(l_1), l'_{m'}(l_2)$, \dots, $l'_{m'}(l_m)$ show that $\theta_1,\mu'_{m'-1}, l'_{m'}(\theta_3)$ are $m$-lined; and
    \item The embeddings $l'_{m'}(l'_1), l'_{m'}(l'_2)$, \dots, $l'_{m'}(l'_{m'-1})$ show that $\mu'_{m'-1}, l'_{m'}(\theta_3),\theta_4$ are $m'$-lined. 
  \end{itemize}
  By induction, $\theta_1,\mu'_{m'-1}, \theta_4$ are $m\cdot 2^{m'-1}$-lined. Putting these together, we get that $\theta_1,\theta_2,\theta_4$ are $m\cdot 2^{m'}$-lined. 
  %
\end{proof}

Let us consider a specific example to see what is happening. When $m=3,m'=2$, we can express our conditions as in Figure~\ref{fig:123-3-lined_234-2-lined}.

  \begin{figure}[htpb]
  \centering
    \drawSimplePatternFigure{0.1}{0.5}{0.5}{0,1,2,-1,0,2,3,-1,0,3,4,-1,1,4,5,-1,1,5,6,-1}    
    \caption{$\theta_1$, $\theta_2$, $\theta_3$ are 3-lined; and $\theta_2$, $\theta_3$, $\theta_4$ are 2-lined.}
    \label{fig:123-3-lined_234-2-lined}
  \end{figure}

Every row represents an embedding, and every circle denotes a cardinal. Two circles in the same column represent the same cardinal. 

In Figure \ref{fig:123-3-lined_234-2-lined}, the first, second, fifth, and seventh columns represent $\theta_1,\theta_2,\theta_3,\theta_4$ respectively, and the sixth column corresponds to the cardinal $\mu'_{m'-1}= \mu'_1$ in the proof above. The first three rows form the witness that $\theta_1,\theta_2,\theta_3$ are $3$-lined, and the last two rows form the witness that $\theta_2,\theta_3,\theta_4$ are $2$-lined. We use $l_1,l_2,l_3,l'_1,l'_2$ to denote the embeddings represented by those 5 rows.

Let us consider replacing $l'_2$ in the sequence with $l'_2(l_1)$, $l'_2(l_2)$, $l'_2(l_3)$, $l'_2(l'_1)$, that is, applying $l'_2$ to each of the other embeddings. Then the resulting sequence of embeddings can be represented by the diagram in Figure~\ref{fig:l2_application}. In this diagram, the first, second and fifth columns correspond to $\theta_1$, $\theta_2$, $\theta_3$, and the last column corresponds to~$\theta_4$. 

\drawSimplePattern{0.1}{0.5}{0.5}{0,1,2,-1,0,2,3,-1,0,3,4,-1,1,4,5,-1,0,5,6,-1,0,6,7,-1,0,7,8,-1,5,8,9,-1}[Result of applying $l'_2$ to the other embeddings.][fig:l2_application]

Then we apply the induction hypothesis to the first 4 rows and the last 4 rows in Figure~\ref{fig:l2_application}.

Intuitively, in the diagram, $l(k)$ denotes moving the circles of $k$ along $l$. In this way, Figure~\ref{fig:l2_application} can be considered as the result of applying the last row to the four rows above it, in Figure~\ref{fig:123-3-lined_234-2-lined}.

Lemma~\ref{lem:lined_composition} has an immediate consequence:

\begin{lemma}
\label{aaab}
  If $\theta_1,\theta_2,\theta_3$ are $m$-lined, and there is $k\in \Iter(j)$ such that $k:\theta_1\mapsto\theta_2\mapsto\theta_3\mapsto\theta_4$, then $\theta_1,\theta_2,\theta_4$ are $m\cdot 2^m$-lined.
\end{lemma}

\begin{proof}
  Suppose that $l_1$, $l_2$, \dots, $l_m$ witness that $\theta_1,\theta_2,\theta_3$ are $m$-lined. Then $k(l_1), k(l_2)$, \dots, $k(l_m)$  witness that $\theta_2,\theta_3,\theta_4$ are $m$-lined as well; we apply Lemma~\ref{lem:lined_composition}. 
\end{proof}


For example, when $m=3$, the conditions in this lemma can be represented by Figure~\ref{fig:lemma_aaab_cond}. The first, second, fifth, and sixth columns represent the cardinals $\theta_1, \theta_2, \theta_3, \theta_4$ respectively.

\drawSimplePattern{0.1}{0.5}{0.5}{0,1,2,-1,0,2,3,-1,0,3,4,-1,0,1,4,5,-1}[Representation of the conditions in Lemma~\ref{aaab} for $m=3$.][fig:lemma_aaab_cond]

Applying the last row to the previous rows, we get Figure~\ref{fig:lemma_aaab_result}.

\drawSimplePattern{0.1}{0.5}{0.5}{0,1,2,-1,0,2,3,-1,0,3,4,-1,1,4,5,-1,1,5,6,-1,1,6,7,-1}[Result of applying the last row's embedding to the previous rows.][fig:lemma_aaab_result]

And we can apply Lemma~\ref{lem:lined_composition} on this diagram and get that $\theta_1, \theta_2, \theta_4$ are $3 \cdot 2^3 = 24$-lined.

\smallskip

This analysis allows us to give a relatively weak lower bound on the growth rate of the function~$F$ mentioned in the introduction. We use the following lemma, which is \cite[Lemma~2]{Dougherty}. Recall that $\gamma_0 < \gamma_1 < \dots < $ is the increasing enumeration of $\crit^*(j)$, so that $\kappa_n = \gamma_{F(n)}$. 

\begin{lemma}[\cite{Dougherty}] \label{lem:doughery_2_n_lemma}
  If $\theta_0,\theta_1,\theta_2$ are $m$-lined, then $\theta_2\ge \gamma_{2^m}$.
\end{lemma}

Notice that 
 \[
 j_{(2)}:\kappa_1\mapsto\kappa_2\mapsto\kappa_3\mapsto \cdots
 \]
 and 
  \[
  j(j_{(2)}):\kappa_2\mapsto\kappa_3\mapsto\cdots 
  \]
 Repeatedly applying $j_{(2)}$ or $j(j_{(2)})$ on $j$ we see that for any $n\in \omega$ there are $j_n,j'_n\in \Iter(j)$ such that 
  \[
  j_n:\kappa_0\mapsto\kappa_{n+1}\mapsto\kappa_{n+2}
  \]
  and
  \[
  j'_n:\kappa_0\mapsto\kappa_1\mapsto\kappa_{n+2}\mapsto\kappa_{n+3}.
  \]
So $\kappa_0,\kappa_1,\kappa_{n+2}$ are $n$-lined, so by Lemma~\ref{aaab}, $\kappa_0,\kappa_1,\kappa_{n+3}$ are $n\cdot 2^n$-lined, so by Lemma~\ref{lem:doughery_2_n_lemma}, $\kappa_{n+3}\ge \gamma_{2^{n\cdot 2^n}}$. So $F(n+3)\ge 2^{n\cdot 2^n}$.

\subsubsection*{Basic Laver patterns and systems}

Now we start the definition of basic Laver patterns.

For a sequence $s$, we use $|s|$ to denote the number of entries in $s$. We write $s=(s_1,...,s_n)$. We use $s_{-k}$ to denote $s_{|s|+1-k}$, that is, the $k$th element from the end; $s_{-1}$ is the last element of $s$. This notation will be used also when $s$ is a sequence of sequences; in that case, $s_{i,k}$ is the $k$th entry of the $i$th sequence $s_i$. Similarly, $s_{i,-1}=s_{i,|s_i|}$ is the last entry of the $i$th sequence $s_i$, $s_{i,-2} = s_{i,|s_i|-1}$ is the second to last entry, and so on. 

Now we define a combinatorial structure that will describe some of the critical points we will be interested in.

\begin{definition}
    We say that a pair $p = (s, l)$ is a \textit{basic Laver pattern}, or \textit{blp} for short, if for some integer $n \ge 2$ (called the \textit{length} of $p$), we have:

    \begin{enumerate}
        \item $s = (s_1, s_2, \dots, s_n)$ is a sequence such that for each $i$, $s_i$ is a strictly increasing sequence of natural numbers, $|s_i| \ge 3$, and the last two entries of $s_i$ are $i, i+1$. Moreover, $s_1 = (0, 1, 2)$ and $s_2 = (0, 1, 2, 3)$. The sequence $s_i$ is called the \emph{$i$th row} of $p$. 

        \item $\ell = (\ell_1, \ell_2, \dots, \ell_n)$, where $\ell_i \in \omega$. When $|s_i|$ is odd, $\ell_i=\frac{|s_i|-1}{2}$. When $|s_i|=4,\ell_i=1$. When $|s_i|>4$ is even, $\ell_i$ is either $\frac{|s_i|}{2}$ or $\frac{|s_i|}{2}-1$. $\ell_i$ is called the \textit{step length} of the row $i$.
    \end{enumerate}

    If $|s_i|>3$ is odd, then we say that the row $i$ is \textit{suitable}. 
\end{definition}

Given a blp $p$, we can draw the pattern to make it clearer.

For example, for $s = ((0, 1, 2), (0, 1, 2, 3), (0, 1, 2, 3, 4), (0, 1, 2, 4, 5))$, $\ell = (1, 1, 2, 2)$, we can draw $p$ as in Figure~\ref{fig:blp_{ep}xample}.

\drawSequence{0.1}{0.5}{0.5}{0,1,2,-1,0,1,2,3,-1,0,1,2,3,4,-2,0,1,2,4,5,-2}[A basic Laver pattern with $s = ((0, 1, 2), \dots)$ and $\ell = (1, 1, 2, 2)$.][fig:blp_{ep}xample]

The step lengths $\ell_i$ indicates that the corresponding embedding shifts the ordinal $\ell_i$ places to the right rather than only one. The definition of basic Laver system below will make this clear.

\begin{notation} \label{notation:python_notation}
When there are several blps, we use, for example, $p.s_2, p.l_3$ to denote those $s_2,l_3$ of $p$. We use $p.n$ to denote the length~$n$ of $p$. We found that this python-like notation is easier to use in this paper.
\end{notation}

To define what is a realization of a blp, we use the following notation.

For an elementary embedding \( k : V_\lambda \to V_\lambda \), we write
\[
k : (m)\theta_0 \mapsto \theta_1 \mapsto \theta_2 \mapsto \dots \mapsto \theta_n
\]
if \(\theta_0 = \mathrm{crit}(k)\) and, for \(n' \le n-m\), we have \(\theta_{n'+m} = k(\theta_{n'})\). This generalizes the notation above, which is the case $m=1$. If \(m=1\) and \(n>1\), we may omit \((m)\) as we did above. 

The following lemma is straightforward and used frequently. It generalizes Lemma~\ref{lem:l_k_application}.

\begin{lemma} \label{lem:lk_application_general}
 If
  \[
k : (m)\theta_0 \mapsto \theta_1 \mapsto \dots \mapsto \theta_n,
\]
then for another embedding \(l\),
\[
l(k) : (m) \, l(\theta_0) \mapsto l(\theta_1) \mapsto \dots \mapsto l(\theta_n).
\]
\end{lemma}

\begin{definition}
    A \textit{basic Laver system} (bls for short), consists of the following data: 
    \begin{itemize}
      \item a blp $p$,
      \item a positive integer $m$,
      \item a sequence of ordinals $\theta'_0,...,\theta'_{m}$,
      \item a sequence of embeddings $l_1,l_2,...,l_{m}$,
      \item a sequence of ordinals $\theta_0,\theta_1,...,\theta_{n+1}$ (where $n=p.n$ is the length of $p$), and
      \item embeddings $j_1,...,j_{n}\in \mathrm{Iter}(j)$,
    \end{itemize}
    such that:
    \begin{enumerate}
      \item $\theta'_0,...,\theta'_{m}$ and $l_1,l_2,...,l_{m}$ witness that $\theta_0,\theta_1,\theta_2$ are $m$-lined,
      \item For each $i\ge 1$, $j_i:(\ell_i)\theta_{s_{i,1}}\mapsto \theta_{s_{i,2}} \mapsto \dots \mapsto\theta_{s_{i,|s_i|}}$. In particular, $\mathrm{crit}(j_i)=\theta_{s_{i,1}}$.
    \end{enumerate}
\end{definition}

When $P$ is a bls, we use $P.p$ to denote its blp $p$. We say that $P$ is a \emph{realization} of $P.p$. We use the pyhton-like notation will all components of a bls, for example, we write $P.j_{2}, P.\theta'_1, P.l_1'$ to denote $j_2,\theta'_1,l'_1$ of $P$. We inherit properties from~$p$, for example. $P.p.s_2$ denotes the $s_2$ of the blp~$p$ that~$P$ realises. However, we often omit details if they are clear from context, so we can write $P.s_2$, or just~$s_2$, $\theta'_1$, etc.

\begin{definition}
  Let $p$ be a blp with length $p.n\ge 3$. 
  \begin{itemize}
    \item We define the blp $p.del$ by deleting the last row of $p$.
    \item If $P$ is a bls of length $n\ge 3$, we define $P.del$ to be the realization of $p.del$ obtained from~$P$ by deleting the last cardinal $\theta_{n+1}$ and the last embedding $j_n$.
  \end{itemize}
  When $2\le n'\le n$, we define $P|n'$ by applying $.del$ to $P$ for $n-n'$ times (intuitively, only keep the first $n'$ rows).
\end{definition}

\begin{lemma} \label{lem:moving_any_theta_i_to_theta_j}
  For any bls $P$, for any $0\le i<i'\le P.n+1$, there is $l\in \mathrm{Iter}(j)$ such that $l:(1)P.\theta_i \mapsto P.\theta_{i'}$.
\end{lemma}

\begin{proof}
  We use induction on $n$. Consider $P.del$; we only need to consider the case $i'=n+1$. If $k:(1)\xi_1\mapsto\xi_2$ and $k':(1)\xi_2\mapsto\xi_3$, then $k'(k):(1)\xi_1\mapsto\xi_3$. By the induction hypothesis, we only need to consider the case $i=n$. Assume $j_{n}$ maps $\theta_u,\theta_v$ to  $\theta_n$, $\theta_{n+1}$ (We have $n=s_{n,-2}$, $n+1=s_{n,-1}$, so we only need to take $u=s_{n,-2-\ell_n}$, $v=s_{n,-1-\ell_n}$). Applying $j_{n}$ to an embedding $k:(1)\theta_{u}\mapsto\theta_{v}$ (which exists by the induction hypothesis) yields the desired result.
\end{proof}

\subsubsection*{The copying operation}

Now we define the copying operation for a blp. Notice that in the diagram of a blp, in the $k$-th row, the last two circles are in the $(k+1)$-th and $(k+2)$-th columns. This gives the diagram a staircase shape. Intuitively, the copying operation tries to choose the set of rows to which the last row applies (by choosing $a,b$ in the later definition), and replaces the last row with these applications, so that we can keep this staircase shape. 

We need a little operation for the definition of copying. Given two increasing finite sequences $s,t$ of natural numbers (as usual, indexes start from 1), and a positive integer $\ell$, we define a partial map $ap(s,t,\ell)$ which outputs an increasing natural number sequence. Intuitively, the resulting sequence realizes the embedding that comes from applying the $t$-embedding to the $s$-embedding, where $\ell$ is the step length of the $t$ row. 

There are cases in which no application is possible. The function $ap(s,t,\ell)$ is only defined if:
\begin{itemize}
  \item  $t_{-\ell-2}>0$ and $|t|\ge \ell+2$;
  \item every entry~$x$ of~$s$ satisfies $x\le t_{-\ell-1}$;
  \item for every entry~$x$ of~$s$, if $t_1\le x<t_{-\ell-2}$ then~$x$ is an entry of~$t$. 
\end{itemize}
When these conditions hold, $ap(s,t,\ell)$ is the sequence of length $|s|$ defined as follows:
\begin{itemize}
  \item if $s_i<t_1$, then $ap(s,t,\ell)_i=s_i$;
  \item   if $s_i=t_j<t_{-\ell-2}$, then $ap(s,t,\ell)_i=t_{j+\ell}$; 
  \item if $s_i\ge t_{-\ell-2}$, then $ap(s,t,\ell)_i=t_{-2}+s_i-t_{-\ell-2}$.
\end{itemize}

For example, if $s = (1,2,3,7,8)$ and $t = (2,3,6,8,9,10)$, then $ap(s,t,2) = (1,6,8,10,11)$. On the other hand, if $s' = (1,2,3,4)$ then $ap(s',t,2)$ is undefined, because there ``isn't enough room'' for~$\theta_4$ to go to.

\begin{definition}
    Let $p$ be a blp. We define the copyability of $p$, and the blp $p.Copied$ when $p$ is copyable. Let $n = p.n$ be the length of~$p$. If $n=2$, then $p$ is not copyable.

    Let $a = s_{n, - \ell_n - 2}$ and $b = s_{n, - \ell_n - 1} - 1$. Then $q = p.Copied$ will be the blp we obtain from~$p$ by replacing the last row $s_n$ of~$p$ by its application to all lines between~$a$ and~$b$ (inclusive); the numbers~$a$ and~$b$ are chosen so that the resulting shape is the desired staircase. In detail: $p.Copied$ is the blp~$q$ defined by:

    \begin{itemize}
      \item $q|n-1=p.del$ (we keep the first $n-1$ rows);
      \item the length $q.n$ of~$q$ is $n+b-a$;
      \item For $i=0,1,\dots, b-a$, $q.\ell_{n+i}=\ell_{a+i}$ and $q.s_{n+i}=ap(s_{a+i},s_{n},\ell_n)$. 
    \end{itemize}
    If~$ap$ is undefined on any of the triples above, then $p$ is not copyable. 
\end{definition}

We observe that the requirements on the application operation ensure that when $p$ is copyable, then $p.Copied$ is indeed a blp. 

\begin{remark} \label{rmk:applying_a_row_of_length_3}
  If the last row $s_n$ of~$p$ has length~3, and its first entry $s_{n,1}$ is nonzero, then $p$ is copyable, as $a = s_{n,1}$ (the step length~$\ell_n$ is necessarily~1). Since $s_{n,2} = n$ in this case, the rows to be copied would be $a,a+1,\dots, n-1$. The effect on each row would be to preserve the entries smaller than~$a$, and increse the other entries by $n-a$. 
\end{remark}

\medskip

We illustrate $p.Copied$ with two examples. The beautiful color \textcolor{TianyiBlue}{\#66CCFF} is used to highlight the rows that will be copied in the copying operation.

If $p$ is as in Figure~\ref{fig:copy_ex1_p}, then $p.Copied$ is as in Figure~\ref{fig:copy_ex1_copied}.

\drawColoredSequence{0.1}{0.5}{0.5}{0,1,2,-1,0,1,2,3,-1,0,1,2,3,4,-2,0,1,2,3,4,5,-3}[A copyable blp $p$.][fig:copy_ex1_p]

\drawColoredSequence{0.1}{0.5}{0.5}{0,1,2,-1,0,1,2,3,-1,0,1,2,3,4,-2,3,4,5,-1}[The resulting blp $p.Copied$.][fig:copy_ex1_copied]

If $p$ is as in Figure~\ref{fig:copy_ex2_p}, then $p.Copied$ is as in Figure~\ref{fig:copy_ex2_copied}.

\drawColoredSequence{0.1}{0.5}{0.5}{0,1,2,-1,0,1,2,3,-1,0,1,2,3,4,-2,0,1,2,3,4,5,-2,2,3,5,6,-1}[Another copyable blp $p$.][fig:copy_ex2_p]

\drawSequence{0.1}{0.5}{0.5}{0,1,2,-1,0,1,2,3,-1,0,1,2,3,4,-2,0,1,2,3,4,5,-2,0,1,3,5,6,-2,0,1,3,5,6,7,-2}[The resulting blp $p.Copied$.][fig:copy_ex2_copied]

However, the blp $p.Copied$ in Fig.~\ref{fig:copy_ex2_copied} is not copyable. Intuitively, this is because we don't know where to put the circles in the third column after moving.

\medskip

If $P$ is a bls that realizes a copyable blp $p = P.p$, then we define  $P.Copied$ to be the unique bls~$Q$ that realizes $q=p.Copied$ (i.e., $Q.p=q$), and the embedding sequence of $Q$ is $j_1,j_2,...,j_{n-1},j_n(j_a),j_n(j_{a+1}),...,j_n(j_{b})$, where $a,b$ are as above, and the witnesses of $m$-linedness remains unchanged. Note that the ``last cardinal'' $P.\theta_{n+1} = P.\theta_{-1}$ of~$P$ is the same as the last cardinal $Q.\theta_{n+b-a+1} = Q.\theta_{-1}$ of $Q$, whereas the second-to-last cardinal $P.\theta_n$ of~$P$ is the same as $Q.\theta_n$, indeed $P.\theta_i = Q.\theta_i$ for $i=0,1,\dots, n$;  the new cardinals $Q.\theta_{n+1}$, $Q.\theta_{n+2}$, \dots, $Q.\theta_{n+b-a}$ are all between $P.\theta_n$ and $P.\theta_{n+1}$. 

We record a fact that we will use later:

\begin{remark} \label{rmk:last_embedding_of_P_copied}
  Suppose that~$P$ is a copyable bls (that is, it realizes a copyable blp $p = P.p$). Then:
  \begin{enumerate}
    \item $P.Copied.\theta_{-1} = P.\theta_{-1}$ (the last cardinals of both bls's are the same); 
    \item The last embedding $P.Copied.j_{-1}$ of $P.Copied$ is the result of applying $P.j_{-1}$ (the last embedding of~$P$) to an embedding $k\colon V_\lambda\to V_\lambda$. 
  \end{enumerate}
\end{remark}

\subsubsection*{Limit types, successor types, and predecessors}

\begin{definition} \label{def:limit_and_successor_types} 
Let~$p$ be a blp; let $n = p.n$ be the length of~$p$.
\begin{enumerate}
  \item We say that $p$ is of \emph{limit type} if the first three entries of the last row $s_n$ of~$p$ are $0,1,2$, $|s_n|=6$, and the step length $\ell_n$ of the last row is~3.

  \item We say that $p$ is of \emph{successor type} if the first three entries of $s_n$ are $0,1,2$, and $|s_n|=5$ (so $s_n = (0,1,2,n,n+1)$, and $\ell_n=2$).

  \item The \emph{zero blp} is the unique blp of length~2. 

  \item If~$p$ is neither of limit type nor of successor type, and is not the zero blp, then we say that~$p$ is of \emph{transient type}.
\end{enumerate}
\end{definition}

We define an operation taking a blp~$p$ of limit or successor type, and an integer~$m$, to a blp $p.E(m)$. Very roughly, if~$p$ is of limit type, then $p.E(m)$ will correspond to the $m$th ordinal in a fundamental sequence for the ordinal represented by~$p$. This is very imprecise intuition though. 

For the following definition, observe that if $p$ is of successor type, then $s_{n,-3} = 2$ (recall that $s_{n,-3}$ is the third-to-last entry of the last row~$s_n$) ; if $p$ is of limit type, then $2<s_{n,-3}<n$. Letting $a = s_{n,-3}$, what we will soon use is that in either case, 
if~$P$ is a realization of~$p$, then $j_n(\theta_0) = \theta_a$, $j_n(\theta_1) = \theta_n$, and $j_n(\theta_2) = \theta_{n+1}$. 

\begin{definition} \label{def:operation_E}
  Suppose that~$p$ is a blp of either limit or successor type, of length~$n$.  Let $a=s_{n,-3}$. We define $p.E(m)$ by recursion on~$m$. 
  \begin{itemize}
    \item $p.E(0) = p.del$.

    \item Suppose that $p.E(m)$ has been defined, and has length~$n$. Let $q_{m+1}$ be the blp obtained from $p.E(m)$ by adding one row, $(a,n'+1,n'+2)$ where $n' = p.E(m).n$ is the length of~$p.E(m)$; the step length is necessarily~1. We let $p.E(m+1) = q_{m+1}.Copied$.\footnote{Note that by Remark~\ref{rmk:applying_a_row_of_length_3}, $p.E(m)$ is always well-defined.}
  \end{itemize}
\end{definition}

For example, if $p$ is as in Figure~\ref{fig:e_op_p}, then $p.E(2)$ is as in Figure~\ref{fig:e_op_result}. 

\drawSequence{0.1}{0.5}{0.5}{0,1,2,-1,0,1,2,3,-1,0,1,2,3,4,-2,0,1,2,3,4,5,-3}[A blp $p$ for demonstrating the $E(m)$ operation; $a=3$.][fig:e_op_p]

\drawSequence{0.1}{0.5}{0.5}{0,1,2,-1,0,1,2,3,-1,0,1,2,3,4,-2,0,1,2,4,5,-2,0,1,2,5,6,-2,0,1,2,6,7,-2}[The resulting blp $p.E(2)$.][fig:e_op_result]

\begin{remark} \label{rmk:the_length_of_the_E_blp}
   Let $n' = p.E(m).n$ be the length of $p.E(m)$. The rows copied when passing from $q_{m+1}$ to $p.E(m+1)$ are rowss $a$, $a+1$, \dots, $n'$ of $p.E(m)$ (the last $n'-a+1$ rows of $p.E(m)$), so the length of $p.E(m+1)$ is $2n' -a +1$. By induction, we see that the length of $p.E(m)$ is $2^m(n-a)+ (a-1)$. 
\end{remark}

\medskip

Suppose that $p$ is a blp of successor or limit type; let~$P$ be a realization of~$p$. Again let $a = p.s_{n,-3}$. Let $m = P.m$; let $l_i = P.l_i$ so $l_1$, $l_2$,\dots, $l_m$ are the witnesses that $\theta_0,\theta_1,\theta_2$ are $m$-lined (with cardinals $\theta_1 = \theta_0'<\theta_1'<\cdots <\theta_m'= \theta_2$). For $i=0,\dots, m$ let $\mu_i = j_n(\theta_i')$. As mentioned above, since $j_n(\theta_0) = \theta_a$, $j_n(\theta_1) = \theta_n$, and $j_n(\theta_2) = \theta_{n+1}$, so $\theta_n = \mu_0 <\mu_1 < \cdots <\mu_m = \theta_{n+1}$; and for $i=1,\dots, m$, 
\[
  j_n(l_i) \colon \theta_a \mapsto \mu_{i-1}\mapsto \mu_i. 
\]

This enables the following definition. 

\begin{definition} \label{def:E_operation_on_bls}
  Suppose that $p$ is a blp of successor or limit type; let~$P$ be a realization of~$p$. Let $m = P.m$, $n = p.n$,  and $l_1,\dots, l_m$, $\mu_0,\dots, \mu_m$ as discussed above. 

  We define sequences $P.E(0),P.E(1),\dots, P.E(m)$ and $Q_1,Q_2,\dots, Q_m$ of bls's with~$P.E(i)$ realizing $p.E(i)$ and $Q_i$ realizing the blp~$q_i$ of Definition \ref{def:operation_E}. We ensure that the last cardinal $P.E(i).\theta_{-1}$ of $P.E(i)$ is~$\mu_i$. The definition is as follows:
  \begin{itemize}
    \item $P.E(0) = P.del$.
    \item Given $P.E(i)$, we let $Q_{i+1}$ be the realization of~$q_{i+1}$ where the last embedding $Q_{i+1}.j_{-1}$ is $j_n(l_{i+1})$. 
    \item We let $P.E(i+1) = Q_{i+1}.Copied$. 
  \end{itemize}
  Finally we let $P.E = P.E(m)$. 
\end{definition}

Note that indeed, since the last cardinal of $P.E(i)$ is~$\mu_i$, the last cardinal of $Q_{i+1}$ and of $P.E(i+1)$ will be $\mu_{i+1}$, as $j_n(l_{i+1})$ maps $\mu_i$ to $\mu_{i+1}$. 

When $i>0$, since the last embedding $Q_i.j_{-1}$ of $Q_i$ is $j_n(l_i)$ (with critical point~$\theta_a$), so the last emebdding $P.E(i).j_{-1}$ of $P.E(i)$ is the result of applying $j_n(l_i)$ to the last embedding of $P.E({i-1})$. So the last embedding of $P.E$ is 
\[
  j_n(l_m) \Big( j_n(l_{m-1}) \Big( \cdots \Big(  j_n(l_1)(j_{n-1})    \Big) \cdots  \Big)  \Big).
\]

We record facts that we will use later.

\begin{lemma} \label{lem:facts_about_P.E}
  Let~$P$ be a realization of a blp of limit or successor type. 
  \begin{enumerate}
    \item $P.E.\theta_{-1} = P.\theta_{-1}$, that is, the last cardinals of~$P$ and of $P.E$ are the same. 

    \item There are elementary embeddings $k_0$ and~$k_1$ such that $\crit P.j_{-1}(k_0) < P.\theta_{-1}$ and
    \[
      P.E.j_{-1}  = (P.j_{-1}(k_0))(k_1).
    \]
    That is, the last embedding of~$P.E$ is obtained from the last embedding of~$P$ by twice applying to some embedding, and the critical point of the intermediate embedding as well has critical point below the last cardinal of~$P$. 
  \end{enumerate}
\end{lemma}




\subsubsection*{The completion and modification operations}

Now we define the completion operation for a suitable row of a blp. Recall that a row~$s_i$ is \emph{suitable} if $|s_i|\ge 5$ is odd. Note that when $s_i$ is suitable, a realizing embedding mapps the first entry $s_{i,1}$ to the middle entry $s_{i,\ell_i+1}$, and that middle entry to the last entry $s_{i,-1} = i+1$. The second-to-middle entry $s_{i,\ell_i}$ is mapped to the second-to-last entry $s_{i,-2} = i$.

Let $s_i$ be a suitable row. Suppose that $s_{i,\ell_i}< s_{i,\ell_i+1}$. If $s_{i,\ell_i} < x < s_{i,\ell_i+1}$ then a corresponding embedding will map $\theta_x$ to a cardinal strictly between $\theta_{i}$ and $\theta_{i+1}$, i.e., between the last two entries. We want to enrich the row by filling in the ``missing circles'' between~$i$ and $i+1$, that are currently not named by any cardinal. In order to keep the staircase shape, we need to add several rows. Say that we want to name~$u$ many new cardinals between $\theta_i$ and~$\theta_{i+1}$. So $\theta_{i+1}$ needs to be renamed to $\theta_{i+u+1}$, and in general, for $j>i$, the old $\theta_j$ becomes the new $\theta_{j+u}$. So we need to replace row~$s_i$ by $u+1$ many rows that ``interpolate'' the passage from $\theta_i$ to $\theta_{i+1}$ through the new cardinals. The new cardinals will be named $\theta_{i+1}$, \dots, $\theta_{i+u}$.

\begin{definition} \label{def:full_comp}
    Let $s_i$ be a suitable row of $p$. 
    
    \begin{enumerate}
    \item  Let $T = \{t_1, t_2, \dots, t_u\}$ be a set of integers satisfying 
    \[
     s_{i,\ell_i} < t_1 < t_2 < \cdots < t_u < s_{i,\ell_i+1}. 
    \]
    We define a blp $q=\mathrm{comp}(p, i, T)$ as follows. 
    \begin{itemize}
      \item The length $q.n$ of~$q$ is $n+u$. 

      \item $q | (i-1) = p|(i-1)$, i.e., for all $i'<i$, the $i'$th row of~$q$ is the same as the $i'$th row of~$p$, with the same step length $\ell_{i'}$. 

      \item For $r=0,1,\dots, u-1$, the $(i+r)$th row $q.s_{i+r}$ of~$q$ is obtained from $s_i$ by inserting $t_1,t_2,\dots, t_r, t_{r+1}$ between $s_{i,\ell_i}$ and $s_{i,\ell_i+1}$, and replacing the last entry~$i+1$ by the sequence $i+1$, $i+2$, \dots, $i+r$, $i+r+1$.  The step length of this row is $q.\ell_{i+r} = p.\ell_i+ r+1$.

      Note that these rows have even length: $|q.s_{i+r}| = |s_i| + 2r+1$. 

      \item We let $q.s_{i+u}$ be the row obtained from~$s_i$ by inserting $t_1,\dots, t_u$ between $s_{i,\ell_i}$ and $s_{i,\ell_i+1}$, and replacing the last entry~$i+1$ by the sequence $i+1$, $i+2$, \dots, $i+u+1$. The step length is $q.\ell_{i+u} = \ell_i+u$. 

      Note that $|q.s_{i+u}| = |s_i|+2u$ is odd. 

      \item For each~$i'$ with $i<i'\le p.n$, the $(i'+u)$th row $q.s_{i'+u}$ of~$q$ is obtained from $p.s_{i'}$ by replacing any entry $x>i$  by $x+u$ (the smaller entries remain unchanged). The step length is $q.\ell_{i'+u} = p.\ell_{i'}$. 
    \end{itemize}

    Note that if $T = \varnothing$, then $\mathrm{comp}(p, i, T) = p$.


     \item Define a sequence $x_0,x_1,\dots$ as follows: 
     \begin{itemize}
       \item $x_0 = s_{i,\ell_i+1}$; 
       \item $x_{r+1}= s_{x_r,-3}$. Note that $x_{r+1}< x_r$ (since $s_{x_r,-2} = x_r$). 
       \item Halt when $x_k\le s_{i,\ell_i}$. 
     \end{itemize}

     Let $T = \{x_{k-1},x_{k-2},\dots, x_1\}$. We define $\mathrm{fullcomp}(p, i) = \mathrm{comp}(p, i, T)$. 
    \end{enumerate}
\end{definition}

We give a simple example. Let~$p$ be the blp described in Figure~\ref{fig:comp_p}; let $i=5$. Note that $|s_i|= 5$ so~$s_i$ is indeed suitable. We have $s_{i,\ell_i} = 1$ and $s_{i,\ell_i+1} = 3$. Since $s_{3,-3} = 2$, we have $x_1 = 2$, so $\mathrm{fullcomp}(p, 5) = \mathrm{comp}(p,5, \{2\})$ is obtained by replacing~$s_5$ by two new rows: in the first, we add the entry~2; in the second we add the entries~2 and $7$. We update the 6th row (the new 7th row) by replacing the entries 6, 7 by 7,8. The result is illustrated in Figure~\ref{fig:comp_result}. If, on the other hand, the 3rd row of~$p$ omitted the entry 2, then we would have $\mathrm{fullcomp}(p,5)= p$. 

\drawSequence{0.1}{0.5}{0.5}{0,1,2,-1,0,1,2,3,-1,0,1,2,3,4,-2,0,1,2,3,4,5,-2,0,1,3,5,6,-2,0,1,3,5,6,7,-2}[A blp $p$ for demonstrating the `fullcomp` operation.][fig:comp_p]

\drawSequence{0.1}{0.5}{0.5}{0,1,2,-1,0,1,2,3,-1,0,1,2,3,4,-2,0,1,2,3,4,5,-2,0,1,2,3,5,6,-3,0,1,2,3,5,6,7,-3,0,1,3,5,7,8,-2}[The resulting blp $\mathrm{fullcomp}(p, 5)$.][fig:comp_result]

There are no difficulties in defining a completion of a bls; after all, the operation just shifts indices so that we can name some new intermediate cardinals; we do not need any new embeddings.

\begin{definition} \label{def:fullcomp_for_bls}
  Let~$P$ be a realization of a blp~$p$; let $s_i$ be a suitable row of~$p$. Let $T = \{t_1< t_2< \cdots <t_u\}$ be a set of integers between $s_{i,\ell_i}$ and $s_{i,\ell_i+1}$. Let $q = \mathrm{comp}(p,i,T)$. We let $\mathrm{comp}(P,i,T)$ be the bls~$Q$ defined as follows. Let $n = p.n$ be the length of~$p$. Let $\theta_0, \dots, \theta_{n+1}$ be the cardinals of~$P$. 
  \begin{itemize}
    \item The cardinals of~$Q$ are 
     \[
     \theta_0,\theta_1,\dots, \theta_i, j_i(\theta_{t_1}), j_i(\theta_{t_2}), \dots, j_i(\theta_{t_u}), \theta_{i+1}, \theta_{i+2},\dots, \theta_{n+1}. 
     \]
     (where $j_i = P.j_i$ is the $i$th embedding of~$P$). 

     \item $Q.m = P.m$ and the witnesses for $m$-linedness are unchanged. 

     \item \begin{itemize}
       \item For $i'<i$, $Q.j_{i'} = P.j_{i'}$. 
       \item For $r = 0,1,\dots, u$, $Q.j_{i+r} = P.j_i$. 
       \item For $i' = i+1,\dots, n$, $Q.j_{i'+u} = P.j_{i'}$. 
     \end{itemize}
  \end{itemize}
\end{definition}



\begin{definition} \label{def:modification_operation}
  Suppose that~$p$ is a copyable blp. Let $i_1< i_2 < \cdots < i_k$ be the indices of the suitable rows of~$p.Copied$. Let $p.M$ be the blp obtained from $p.Copied$ by performing $\mathrm{fullcomp}(-,i_k)$, then $\mathrm{fullcomp}(-,i_{k-1})$, \dots, then $\mathrm{fullcomp}(-,i_{1})$. 

  If~$P$ is a realization of a copyable blp~$p$, then $P.M$ is obtained from $P.Copied$ by applying the same sequence of operations. Hence, $P.M$ is a realization of $p.M$. 
\end{definition}

As usual, we record facts that will be used shortly:

\begin{remark} \label{rmk:last_embedding_of_P.M}
  Suppose that~$P$ is a copyable bls. 
  \begin{enumerate}
    \item  $P.M.\theta_{-1} = P.\theta_{-1}$ ($P.M$ and~$P$ have the same last cardinal);
    \item  The last embedding $P.M.{j_{-1}}$ of $P.M$ is the last embedding of~$P.Copied$. 
  \end{enumerate}
\end{remark}





\subsubsection*{The main theorem}

Now we recursively define a function to estimate the critical points below a bls. It is defined as a partial function, although it is provable that the function is total.

\begin{definition} \label{def:the_f_function}
    We recursively define the partial function $f$ on a pair $(p,m)$, where $p$ is a blp, and $m\in \omega, m>0$, as follows:
    \begin{enumerate}
    \item If~$p$ is the zero blp, then $f(p,m)=m\cdot 2^m$.
    \item Else, if $p$ is not copyable, then $f(p,m)=f(p.del,m)$.
    \item If $p$ is of successor type, then $f(p,m)=(\lambda x.f(p.del,x))^{2^m}(m)$.\footnote{Recall that $g^k$ indicates the composition of~$g$ with itself~$k$ times.}
    \item If $p$ is of limit type, then $f(p,m)=f(p.E(m),m)$.
    \item Else, $f(p,m)=f(p.M,m)$.
    \end{enumerate}
\end{definition}

The definition gives an algorithm for computing $f(p,m)$, when defined, so~$f$ is partial computable. 

\begin{theorem} \label{thm:f_is_lower_bound}
  Suppose that~$P$ is a bls realizing a blp $p = P.p$. Let $m  = P.m$, and let $\theta_0,\theta_1,\dots,\theta_n, \theta_{n+1} = \theta_{-1}$ be the cardinals of~$P$. 

  Then $f(p,m)$ is defined, and $\theta_0,\theta_1,\theta_{-1}$ are $f(p,m)$-lined. 
\end{theorem}

\begin{remark} \label{rmk:f_is_total}
  Theorem~\ref{thm:f_is_lower_bound} implies that~$f$ is a total function: Dougherty \cite{Dougherty} observed that any finite increasing subsequence of the sequence $\kappa_0,\kappa_1,\dots$ of the critical sequence of~$j$ can be realised as the initial segment of a critical sequence of some $j'\in \Iter^*(j)$. It follows that for every blp~$p$ and every $m>0$, there is a realization $P$ of~$p$ with $P.m=m$ and all cardinals $P.\theta_i$ from the $\kappa_i$'s. 
\end{remark}

To prove Theorem \ref{thm:f_is_lower_bound}, we recall a lemma that is Theorem 2.2 of~\cite{Dougherty2}:

\begin{lemma} \label{lem:sup_crit}
    For any infinite sequence $j_1, j_2, \dots$ of elementary embeddings $V_\lambda\rightarrow V_\lambda$, if there exist $j'_1, j'_2, \dots \in \mathrm{Iter}(j)$ such that for each integer $i \ge 1$, $j_i(j'_i) = j_{i+1}$, then $\sup(\crit(j_i)) = \lambda$.
\end{lemma}

Let $R$ be the binary relation on the collection of all elementary embeddings $j\colon V_\lambda\to V_\lambda$ defined as follows:
\begin{itemize}
  \item $j'\,R\,j$ if there is some elementary $k\colon V_\lambda\to V_\lambda$ such that $j' = j(k) $.
\end{itemize}
Then Lemma~\ref{lem:sup_crit} implies that for any $\theta<\lambda$, the restriction of~$R$ to those embeddings with critical point $<\theta$ is well founded. 

\begin{definition} \label{def:smaller_bls}
  We let $P<Q$ be the transitive closure of the following binary relation on bls's:
  \begin{itemize}
    \item $P.\theta_{-1}< Q.\theta_{-1}$, or $P.\theta_{-1}= Q.\theta_{-1}$ and $P.{j_{-1}} = Q.j_{-1}(k)$ for some elementary $k\colon V_\lambda\to V_\lambda$. That is, either the last cardinal of~$P$ is strictly smaller than the last cardinal of~$Q$, or they are equal, and the last embedding of~$P$ stands in relation~$R$ with the last emebdding of~$Q$. 
  \end{itemize}
\end{definition}

Thus, the relation $P<Q$ is a well-founded partial ordering on the collection of all bls's. We say that $P$ is \emph{smaller} than~$Q$.



\begin{lemma} \label{lem:operations_reduce_rank}
  Let~$P$ be a bls. 
  \begin{enumerate}
    \item If $P$ is nonzero then $P.del < P$. 
    \item If $P$ is copyable then $P.M<P$. 
    \item If $P$ is of limit or of successor type, then $P.E<P$.
  \end{enumerate}
  Further, the operations~$M$ and~$E$ preserve the cardinals $\theta_0$, $\theta_1$, and the last cardinal $\theta_{-1}$: $P.M.\theta_i = P.\theta_i$ and $P.E.\theta_i = P.\theta_i$ for $i = 0,1,-1$. 
\end{lemma}

\begin{proof}
  (1) holds because $P.del.\theta_{-1}< P.\theta_{-1}$. 
  For~(3) we use Lemma~\ref{lem:facts_about_P.E}. For~(2) we use Remarks~\ref{rmk:last_embedding_of_P.M} and~\ref{rmk:last_embedding_of_P_copied}. 
\end{proof}


Now we prove the theorem.

\begin{proof}[Proof of Theorem~\ref{thm:f_is_lower_bound}]
    We prove the theorem by induction on the well-founded relation $P<Q$. 
    Let $P$ be a bls; let $p = P.p$ and $m = P.m$. 
    We check the various cases of the definition of~$f$. 

    \smallskip

    For Case 1, the result is Lemma~\ref{aaab}.

    \smallskip

    For Case 2, by induction, the theorem holds for $P.del$. Now $P.del.\theta_{-1}  = P.\theta_{-2}$. Let $r= f(p,m) = f(p.del,m)$. By induction, let  $j_1$, \dots, $,j_r$ witness that $P.\theta_0,P.\theta_1,P.\theta_{-2}$ are $r$-lined. By Lemma~\ref{lem:moving_any_theta_i_to_theta_j}, let $k$ be an elemnetary embedding with critical point $P.\theta_{-2}$ that maps $P.\theta_{-2}$ to $P.\theta_{-1}$. Then the sequence 
    \[
      j'_1, j'_2,\dots, j'_{r-1}, k(j'_r)
    \]
    witness that $P.\theta_0,P.\theta_1,P.\theta_{-1}$ are $r$-lined. 

    \smallskip

    Cases 4 and 5 follow by the induction hypothesis, using Lemma~\ref{lem:operations_reduce_rank}.

    \smallskip

    The remaining case is Case~3. The notation involved is a little cumbersome, so we start by giving an example, that illustrates the general situation. Suppose that~$p$ is as in Figure~\ref{fig:case3_p}. Suppose that $m=2$. 


    \drawSequence{0.1}{0.5}{0.5}{0,1,2,-1,0,1,2,3,-1,0,1,2,3,4,-2,0,1,2,3,4,5,-3,0,1,2,5,6,-2}[The blp for the bls $P$ in the example for Case 3.][fig:case3_p]

    Since $p$ is of successor type (the last row is $(0,1,2,n,n+1)$, in this case $(0,1,2,5,6)$), when applying the $E$ operation we have $a=2$ (in the notation of Definition~\ref{def:operation_E}). This means that at each step (from $p.E(0)$ to $p.E(1)$ to $p.E(2)$) the rows that are copied are all but the first one, achieving a ``doubling effect'' (see Remark~\ref{rmk:the_length_of_the_E_blp}), and the effect on each copied row is to preserve the first two columns, and shift the rest by a constant amount that achieves the stair-case pattern. Thus, the blp $p.E(2)$ is as depicted in Figure~\ref{fig:case3_pe}; since $m=2$, the bls $P.E$ realizes $p.E(2)$. 

    \drawSequence{0.1}{0.5}{0.5}{0,1,2,-1,0,1,2,3,-1,0,1,2,3,4,-2,0,1,2,3,4,5,-3,0,1,5,6,-1,0,1,5,6,7,-2,0,1,5,6,7,8,-3,0,1,8,9,-1,0,1,8,9,10,-2,0,1,8,9,10,11,-3,0,1,11,12,-1,0,1,11,12,13,-2,0,1,11,12,13,14,-3}[The blp $p.E(2)$ from the example.][fig:case3_pe]

    For this example, let $\theta_i = P.E.\theta_i$ denote the cardinals of $P.E$. Since $P.E(0) = P.del$ realizes $p.E(0)= p.del$ with $m=2$, and $P.E(0)<P$, by induction, the cardinals $(\theta_0,\theta_1,\theta_5)$ are $f(p.del,2)$-lined. 

    Let $k_1$ be an embedding with $k_1: \theta_0\mapsto \theta_1 \mapsto \theta_5$ (apply an embedding moving $\theta_2$ to $\theta_5$ to the first embedding of~$P$, using Lemma~\ref{lem:moving_any_theta_i_to_theta_j}). Then this embedding, together with the embeddings given by row $5,6,7$ of Figure~\ref{fig:case3_pe}, give a bls~$Q_0$ that realizes $p.del$ (with $Q.\theta_2 = \theta_5$), and we can set $Q.m = f(p.del,2)$. Since $Q<P$ (the last cardinal of~$Q$ is $P.E.\theta_8$, which is smaller than the last cardinal of~$P$), by induction, $(\theta_0,\theta_1,\theta_8)$ are $f(p.del,f(p.del,2))$-lined. 

    Note that our argument is not just restricted to $P.E(1)$; we actually showed that if $R$ is any bls realizing $p.E(1)$ with $R<P$, then $(R.\theta_0,R.\theta_1,R.\theta_{-1})$ are $f(p.del,f(p.del,R.m))$-lined. We can take the second half of $P.E$ (rows 8 to 13), and with an embedding with sequence $\theta_0\mapsto \theta_1\mapsto \theta_8$ get a realization $Q_1$ of $p.E(1)$ with $Q_1.m = f(p.del,f(p.del,2))$. We have $Q_1<P$ since $Q_1$ and~$P$ have the same last cardinal, But the last embedding of~$Q_1$ is the last embedding of~$P.E$, and $P.E<P$. Applying the inductive hypothesis, we obtain the desired result. 


    \medskip

    We now give the formal details. Let~$p$ have successor type; as above, let $P$ realize~$p$, and let $m = P.m$; let
    \[
    g(x):=f(p.\mathrm{del},x).
    \]
    Recall the sequence of bls's $P.E(i)$ for $i=0,\dots, m$ given in Definition~\ref{def:E_operation_on_bls}; $P.E = P.E(m)$. We start with $P.E(0) = P.del$, and for each $i<m$, $P.E(i+1)$ is obtained from $P.E(i)$ by applying the embedding $j_n(l_{i+1})$ (corresponding to a row $(2,n_i,n_{i+1})$) to all but the first row of $P.E(i)$. 

    We prove the following by induction on $i=0,1,\dots, m$. 

    \begin{description}
    \item[$\mathbf{H}(i)$:] If~$Q$ is any bls that:
    \begin{itemize}
      \item is smaller than~$P$, and 
      \item realizes $p.E(i)$, 
    \end{itemize}
    then the triple $(Q.\theta_0,Q.\theta_1,Q.\theta_{-1})$ is $g^{2^i}(Q.m)$-lined.
    \end{description}

    The point is that $\mathbf{H}(i)$ applies no matter what $Q.m$ is, we do not restrict to $Q.m = i$, or $Q.m = P.m$, or any other restriction. 

    The statement $\mathbf{H}(m)$ suffices to prove Case~3, by taking $Q= P.E$. Note that $P$ and $P.E$ have the same last cardinal, and that $P.E<P$, so $\mathbf{H}(m)$ applies. 

    \smallskip

    We prove $\mathbf{H}(i)$ by induction on~$i$. 

    \smallskip

    \emph{Base case: $i=0$.}
    Here $p.E(0)=p.\mathrm{del}$. Take any $Q$ realizing $p.\mathrm{del}$, that is smaller than $P$. By induction, since $P.del<P$, 
    \[
    (Q.\theta_0,Q.\theta_1,Q.\theta_{-1})\ \text{is}\ f(p.\mathrm{del},Q.m)=g(Q.m)\text{-lined},
    \]
    i.e.\ $\mathbf{H}(0)$ holds.

    \smallskip

    \emph{Inductive step: $i\to i+1$.}
    Fix a bls $Q$ realizing $p.E(i+1)$ that is smaller than~$P$. Let $m' = Q.m$. We must show that
    \[
    (Q.\theta_0,Q.\theta_1,Q.\theta_{-1})
    \]
    is $g^{2^{i+1}}(m')$-lined.

    \textbf{(A) Old block inside $Q$.}
    Let
    \[
    R^{\mathrm{old}}\ :=\ Q\!\mid p.E(i).n
    \]
    be the restriction to the first $p.E(i).n$ rows. Then $R^{\mathrm{old}}$ realizes $p.E(i)$, $R^{\mathrm{old}}.m=m'$, and
    \[
    \nu_i:=R^{\mathrm{old}}.\theta_{-1}\ <\ Q.\theta_{-1};
    \]
    hence, $R^{\mathrm{old}}< Q < P$. 
    Applying $\mathbf{H}(i)$ to $R^{\mathrm{old}}$, we obtain witnesses $w_1,\dots,w_M\in\Iter(j)$ showing that
    \[
    (Q.\theta_0,Q.\theta_1,\nu_i)\ \text{are}\ M\text{-lined},\qquad \text{where }M:=g^{2^i}(m').
    \]

    \textbf{(B) New copied block inside $Q$.}
    Let $R^{\mathrm{new}}$ consist of all rows of~$Q$ with row number larger than $p.E(i).n$, together with a top row emedding with sequence $Q.\theta_0\mapsto Q.\theta_1\mapsto \nu_i$ (which can be obtained by applying to the first embedding of~$Q$ an embedding taking $Q.\theta_2$ to $\nu_i$). As discussed, $p.E(i+1)$ comes from $p.E(i)$ by copying with $a=2$ which applies to all rows of $p.E(i)$ except the top row, and the effect is to increase all entries $\ge 2$ by a fixed constant; see Remark~\ref{rmk:applying_a_row_of_length_3}. 
    This shows that $R^{\mathrm{new}}$ is a realization of $p.E(i)$, with 
    \[
    R^{\mathrm{new}}.\theta_0=Q.\theta_0,\ \ R^{\mathrm{new}}.\theta_1=Q.\theta_1,\ \ 
    R^{\mathrm{new}}.\theta_{-1}=Q.\theta_{-1}, 
    \]
    and $R^{\mathrm{new}}.m=M$; the linedness witnesses of $R^{\mathrm{new}}$ are the $M$ embeddings $w_1,\dots,w_M$ found in (A). The last row of $R^{\mathrm{new}}$ is the same as $Q$, so $R^{\textrm{new}}<P$. 

    \textbf{(C) Reapply the invariant.}
    Applying $\mathbf{H}(i)$ to $R^{\mathrm{new}}$ gives that
    \[
    (Q.\theta_0,Q.\theta_1,Q.\theta_{-1})
    \]
    is $g^{2^i}(M)=g^{2^i}\!\bigl(g^{2^i}(m')\bigr)=g^{2^{i+1}}(m')$-lined, as required.
    This proves $\mathbf{H}(i+1)$, completing the induction. 
\end{proof}


\section{Lower bounds}
\label{sec:lower_bounds}

Now we give the proofs of the estimates presented in the introduction.

\subsubsection*{Small critical points}

First, we recall the analysis of small critical points from~\cite{Dougherty}, specifically for row $n=11$ in Table 2. Recall the embeddings $j_{(m)}$ and the cardinals $\kappa^n_m$ defined in the introduction. Recall that $j : \kappa_0\mapsto \kappa_1 \mapsto \kappa_2 \cdots$. 

\begin{lemma} \label{lem:dougherty_summary} \
  \begin{enumerate}
    \item $j_{(11)}:(2) \,\,\kappa_0 \mapsto \kappa_1 \mapsto \kappa_2\mapsto \kappa_3\mapsto \kappa_2^7 \mapsto \kappa_3^{11}$. 

    \item $j_{(11)}(j_{(11)}):(2)\,\, \kappa_2\mapsto \kappa_3\mapsto \kappa_2^7 \mapsto \kappa_3^{11} \mapsto \kappa_3^{10}$. 
  \end{enumerate}
\end{lemma}

Let $p_{start}$ denote the blp depicted in Figure~\ref{fig:p_start}. 
    \drawSequence{0.1}{0.5}{0.5}{0,1,2,-1,0,1,2,3,-1,0,1,2,3,4,-2,0,1,2,3,4,5,-2,2,3,4,5,6,-2}[The blp $p_{start}$.][fig:p_start]

Lemma~\ref{lem:dougherty_summary} shows that there is a bls~$P_{start}$ realizing the blp $p_{start}$ with $P_{start}.m=1$, $P_{start}.\theta_0 = \kappa_0$, $P_{start}.\theta_1 = \kappa_1$, $P_{start}.\theta_2 = \kappa_2$, $P_{start}.\theta_3 = \kappa_3$, $P_{start}.\theta_4 = \kappa_2^7$, $P_{start}.\theta_5 = \kappa_3^{11}$, $P_{start}.\theta_6 = \kappa_3^{10}$. The witnessing embeddings are $j$, $j$, $j_{(11)}$, $j_{(11)}$ and $j_{(11)}(j_{(11)})$. 

Let $j'=j_{(11)}(j_{(11)})(j_{(11)}(j))(j_{(11)})$. Since
 \[j_{(11)}(j):\kappa_2\mapsto\kappa_3\mapsto\kappa_2^7,\] we get
  \[j_{(11)}(j_{(11)})(j_{(11)}(j)):\kappa_2^7\mapsto\kappa_3^{11}\mapsto\kappa_3^{10};\]  so 
  \[
    j'\colon (2) \,\, \kappa_0 \mapsto \kappa_1\mapsto \kappa_2 \mapsto \kappa_3 \mapsto \kappa_3^{11} \mapsto \kappa_3^{10}. 
  \]
  Thus, the embeddings $j,j,j_{(11)},j_{(11)},j'$ witness that there is a bls~$Q$ realizing the blp described in Figure~\ref{fig:bls_j_prime}, with the same cardinals as $P_{start}$, in particular, with last cardinal $\kappa_3^{10}$, and first two cardinals $\kappa_0$ and~$\kappa_1$.  

\drawSequence{0.1}{0.5}{0.5}{0,1,2,-1,0,1,2,3,-1,0,1,2,3,4,-2,0,1,2,3,4,5,-2,0,1,2,3,5,6,-2}[blp for the system witnessed by $j,j,j_{(11)},j_{(11)},j'$.][fig:bls_j_prime]

Applying the .M operation, we get the bls with blp as in Figure~\ref{fig:bls_after_m}, with the same last cardinal $\kappa_3^{10}$ (and first two cardinals $\kappa_0$ and $\kappa_1$). 

\drawSequence{0.1}{0.5}{0.5}{0,1,2,-1,0,1,2,3,-1,0,1,2,3,4,-2,0,1,2,3,4,5,-2,0,1,2,3,5,6,-3,0,1,2,3,5,6,7,-3,0,1,3,5,7,8,-2}[blp after applying the .M operation.][fig:bls_after_m]

Note that by our construction of the $M$ operation twice, the last two rows are realized by the same embedding, so we can add a pair of nodes in the last row according to the second last row.

This shows that there is a bls $P_{init}$ with last cardinal $\kappa^3_{10}$ (and first cardinals $\kappa_0$ and~$\kappa_1$) that realizes the blp $p_{init}$ described in Figure~\ref{fig:p_init}. 

\drawSequence{0.1}{0.5}{0.5}{0,1,2,-1,0,1,2,3,-1,0,1,2,3,4,-2,0,1,2,3,4,5,-2,0,1,2,3,5,6,-3,0,1,2,3,5,6,7,-3,0,1,2,3,5,6,7,8,-3}[The blp $p_{init}$.][fig:p_init]


\begin{proof}[Proof of Proposition \ref{numberbound}(2)]
Applying Theorem~\ref{thm:f_is_lower_bound} to the bls $P_{init}$, we see that $\kappa_0,\kappa_1,\kappa_3^{10}$ are $f(p_{init},1)$-lined. Now apply Lemma~\ref{lem:doughery_2_n_lemma}.
\end{proof}

\begin{question}
    Give a good evaluation of $f(p_{init},1)$.
\end{question}

Now, by direct computation, $p' := p_{start}.del.M.M.M.M.E(1).E(1).E(1)$ is as in Figure~\ref{fig:p_prime}.

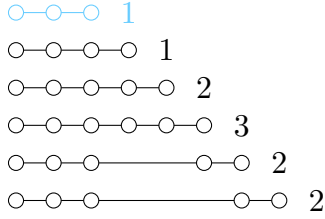
\drawSequence{0.1}{0.5}{0.5}{0,1,2,-1,0,1,2,3,-1,0,1,2,3,4,-2,0,1,2,3,4,5,-3,0,1,2,5,6,-2,0,1,2,6,7,-2}[The blp $p' = p_{start}.del.M.M.M.M.E(1).E(1).E(1)$.][fig:p_prime]

\begin{definition} \label{def:canonical_blps_up_to_omega_plus_2} 
We define blps $p_\alpha$ for ordinals $\alpha<\omega+\omega$. 
  \begin{enumerate}
    \item $p_0$ is the zero blp (the unique blp of length 2).
    \item Suppose that $p_\alpha$ is defined; let $n = p_\alpha.n$ be the length of $p_\alpha$. Then $p_{\alpha+1}$ is the blp obtained from $p_\alpha$ by adding the row $(0,1,2,n+1,n+2)$ at the end. 
    \item $p_\omega$ is the blp depicted in Figure~\ref{fig:p_omega}.
  \end{enumerate}  
\end{definition}


\drawSequence{0.1}{0.5}{0.5}{0,1,2,-1,0,1,2,3,-1,0,1,2,3,4,-2,0,1,2,3,4,5,-3}[The blp $p_\omega$.][fig:p_omega]

Recall that~$m_\alpha$ denote the Steinhaus-Moser functions, defined in the introduction. The definition of this sequence of functions, and the zero, successor, and limit cases in Definition~\ref{def:the_f_function}, give:

\begin{lemma} \label{lem:p_alpha_and_Steinhaus_Moser}
  For all $\alpha<\omega+\omega$, 
   \[
   2^{f(p_{\alpha}, n)} = m_{\alpha}(2^n). 
   \]
\end{lemma}

\begin{proof}[Proof of Proposition~\ref{numberbound}(1)]
  The blp~$p'$ (Figure~\ref{fig:p_prime}) is $p_{\omega+2}$. Recall the bls~$P_{start}$ defined following Lemma~\ref{lem:dougherty_summary}. Then $P_{start}.del$ is a bls realizing $p_{start}.del$ with largest cardinal~$\kappa_3^{11}$, and parameter $m=1$. Applying the sequence of operations $M$ and $E$ to this blp yields a bls $P_{\omega+2}$ with the same largest cardinal (and smallest cardinals $\kappa_0$ and $\kappa_1$) realizing $p_{\omega+2}$ with parameter $m=1$. By Theorem~\ref{thm:f_is_lower_bound}, $(\kappa_0,\kappa_1,\kappa_3^{11})$ are $f(p_{\omega+2},1)$-lined. By Lemma~\ref{lem:doughery_2_n_lemma}, $\crit^*(j)$ contains at least 
  \[
    2^{f(p_{\omega+2},1)} = m_{\omega+2}(2)
  \]
  many cardinals below $\kappa_3^{11}$, i.e., $\kappa_3^{11}\ge \gamma_{m_{\omega+2}(2)}$, as required. 
\end{proof}


\subsubsection*{Peano arithmetic}
\label{subsec:PA}

We now turn to the proof of Theorem~\ref{thm:peano_arithmetic}. Along the way we record a stronger intermediate lower bound which we expect to generalize beyond PA.

Recall the blp $p_{init}$ from Figure~\ref{fig:p_init}. Define
\[
p_{BO}\ :=\ p_{init}.{\mathrm{del}}.{\mathrm{del}}.\,E(1).\,M.\,{\mathrm{del}}.\,M.\,{\mathrm{del}}.{\mathrm{del}}.\,E(1).\,M.\,{\mathrm{del}}.{\mathrm{del}},
\]
see Figure~\ref{fig:p_bo}. Applying the same sequence of operations to the bls $P_{init}$ (constructed after Lemma~\ref{lem:dougherty_summary}) yields a bls $P_{BO}$ realizing $p_{BO}$ with  $P_{BO}.m=1$, and last cardinal
\[
P_{BO}.\theta_{-1}\ <\ \kappa_3^{10}.
\]

\drawSequence{0.1}{0.5}{0.5}{%
0,1,2,-1,%
0,1,2,3,-1,%
0,1,2,3,4,-2,%
0,1,2,3,4,5,-2,%
0,1,2,5,6,-2,%
2,5,6,7,-1,%
0,1,2,5,7,8,-3
}[The blp $p_{BO}$.][fig:p_bo]

\begin{lemma}[simple lining witnesses for $j$]\label{lem:simple_n_lined}
For each $n>0$ there exist embeddings
\[
j,\ j_{(2)}(j),\ (j_{(2)})^2(j),\ \ldots,\ (j_{(2)})^{n-1}(j)\ \in \Iter(j)
\]
witnessing that $\kappa_0,\kappa_1,\kappa_{n+1}$ are $n$-lined.
\end{lemma}

\begin{proof}
Note that $j_{(2)}:\kappa_1\mapsto\kappa_2\mapsto\kappa_3\mapsto\cdots$, hence $j_{(2)}(j):\kappa_0\mapsto\kappa_2\mapsto\kappa_3$, and by repeated application of Lemma~\ref{lem:lk_application_general} we obtain the required chain of embeddings, one for each step from $\kappa_1$ to $\kappa_{n+1}$.
\end{proof}

We shall use the following notion (a special case of Definition~\ref{def:E_operation_on_bls}).

\begin{definition}[Resetting the linedness parameter]
Let $P$ be a bls. We say that a bls $P'$ \emph{resets the linedness of $P$ to $n$} if $P'.p=P.p$, $P'.\theta_i=P.\theta_i$ for $i=0,1,-1$, and $P'.m=n$. By Lemma~\ref{lem:moving_any_theta_i_to_theta_j} and Lemma~\ref{lem:simple_n_lined}, for any given $n$ such a reset exists whenever there are embeddings that line $(P.\theta_0,P.\theta_1,P.\theta_2)$ to $(P.\theta_0,P.\theta_1,P.\theta_{n+1})$; in particular we can reset $P_{BO}$ to any prescribed $n>0$ while preserving the inequality $P_{BO}.\theta_{-1}<\kappa_{n+3}$ described below.
\end{definition}

\begin{lemma}\label{lem:move_kappa_nplus3}
For each $n\ge 1$ there is $h\in \Iter(j)$ with
\[
h:\ (2)\ \kappa_0 \mapsto \kappa_1 \mapsto \kappa_{n+1} \mapsto \kappa_{n+2}.
\]
Consequently, after resetting linedness to $n$ as above, we obtain a bls $\widehat P_{BO}^{(n)}$ with
\[
\widehat P_{BO}^{(n)}.p=p_{BO},\qquad \widehat P_{BO}^{(n)}.m=n,\qquad \widehat P_{BO}^{(n)}.\theta_{-1}<\kappa_{n+3}.
\]
\end{lemma}

\begin{proof}
Let $j^{n}$ denote the $n$-fold composition of $j$. Then $j(j(j^{\,n-1}))$ has the displayed behavior on the critical sequence. Concretely, take the critical sequence of \(j^{\,n-1}\),
\[
  j^{\,n-1}:(2)\,\kappa_0 \mapsto \kappa_1 \mapsto \kappa_{n-1} \mapsto \kappa_n,
\]
and then apply \(j(j(\,\cdot\,))\) to it, yielding
\[
  j\bigl(j(j^{\,n-1})\bigr):(2)\,\kappa_2 \mapsto \kappa_3 \mapsto \kappa_{n+1} \mapsto \kappa_{n+2},
\]
applying it on $j$, we get what was required. 

The reset statement follows by combining Lemma~\ref{lem:simple_n_lined} with the definition above. 
\end{proof}

\begin{corollary}\label{cor:F-lower-bound-BO}
For all $n>0$ we have
\[
F(n+3)\ >\ f(p_{BO},n).
\]
\end{corollary}

\begin{proof}
By Lemma~\ref{lem:move_kappa_nplus3}, $(\kappa_0,\kappa_1,\widehat P_{BO}^{(n)}.\theta_{-1})$ are $f(p_{BO},n)$-lined by Theorem~\ref{thm:f_is_lower_bound}, hence $\kappa_{n+3}>\gamma_{2^{f(p_{BO},n)}}$ by Lemma~\ref{lem:doughery_2_n_lemma}. 
\end{proof}

\paragraph{The $\varepsilon_0$ pattern.}
Define the blp
\[
p_{\varepsilon_0}\ :=\ p_{BO}.{\mathrm{del}}.M.M.{\mathrm{del}.M},
\]
see Figure~\ref{fig:p_epsilon0}. Applying the same operations to $P_{BO}$ yields a bls $P_{\varepsilon_0}$ with $P_{\varepsilon_0}.p=p_{\varepsilon_0}$ and whose first three cardinals are
\[
P_{\varepsilon_0}.\theta_0=\kappa_0,\qquad P_{\varepsilon_0}.\theta_1=\kappa_1,\qquad P_{\varepsilon_0}.\theta_2=\kappa_2.
\]

\drawSequence{0.1}{0.5}{0.5}{%
0,1,2,-1,%
0,1,2,3,-1,%
0,1,2,3,4,-2,%
0,1,2,3,4,5,-2,%
0,1,2,5,6,-2,%
0,1,2,5,6,7,-3,%
5,6,7,8,-1,%
0,1,2,7,8,9,-3
}[The blp $p_{\varepsilon_0}$.][fig:p_epsilon0]

From $P_{\varepsilon_0}$ we extract a shorter pattern $p_{ep}$ as in Figure~\ref{fig:p_{ep}} by the following construction. Let $k:(1)\,\theta_3\mapsto\theta_5$ (Lemma~\ref{lem:moving_any_theta_i_to_theta_j}). Replacing the middle block of embeddings, $j_2,j_3,j_4$, by $k(j_2)$, we obtain a bls $P_{ep}$ with $P_{ep}.p=p_{ep}$ and the same initial cardinals.

\drawSequence{0.1}{0.5}{0.5}{%
0,1,2,-1,%
0,1,2,3,-1,%
0,1,2,3,4,-2,%
0,1,2,3,4,5,-3,%
3,4,5,6,-1,%
0,1,2,5,6,7,-3
}[The blp $p_{ep}$.][fig:p_{ep}]


For $n\ge 1$ we will need explicit blps exhibiting a long terminal block of uniform step length~$3$.

\begin{definition}[The canonical $q_n$]\label{def:q_n}
For $n\ge 1$ let $q_n$ be the blp with $2^n+4$ rows defined as follows.
\begin{itemize}
  \item Rows $1$–$3$ agree with $p_{ep}$ (including step lengths).
  \item For $i=4,\dots,2^n+4$, the $i$-th row is
  \[
  s_i=(0,1,2,i-1,i,i+1),\qquad \ell_i=3.
  \]
\end{itemize}
\end{definition}

\begin{lemma}\label{lem:Qn-exists}
Fix $n\ge 1$. Let $P_n$ be a bls which realizes $p_{ep}$, has $P_n.m=n$, and satisfies $P_n.\theta_{-1}<\kappa_{n+3}$ (Whose existence is clear from the previous discussion). Then there exists a bls $Q_n$ such that:
\begin{enumerate}
\item $Q_n.m=n$;
\item $Q_n.p=q_n$
\item $Q_n.\theta_{-1}\le P_n.\theta_{-1}<\kappa_{n+3}$.
\end{enumerate}
\end{lemma}

\begin{proof}
Let $R^{(0)}:=P_n.E$. Then $R^{(0)}.m=n$ and $R^{(0)}.\theta_{-1}=P_n.\theta_{-1}$, and
the tail consists of \emph{exactly $2^n$} rows of length $4$ with step $1$:
\begin{equation}\tag{$\ast$}\label{eq:tail0}
(3,4,5,6),\ (3,4,6,7),\ \ldots,\ (3,4,2^n\!+\!4,\,2^n\!+\!5).
\end{equation}

We construct $R^{(t)}$ for $t=0,1,\dots,2^n$ so that the following statement $\mathcal I(t)$ holds:

\begin{description}
  \item[$\mathcal I(t)$:] $R^{(t)}.m=n$, $R^{(t)}.\theta_{-1}=P_n.\theta_{-1}$, the first three rows are those of $p_{ep}$, and the tail consists of:
\begin{itemize}
\item $t+1$ canonical length-$6$ rows (step $3$)
\[
(0,1,2,\,j+1,\,j+2,\,j+3)\quad\text{for } j=2,3,4,\dots,t+2;
\]
\item followed by the \textbf{remaining $2^n-t$} rows of length $4$ (step $1$), listed explicitly as
\begin{equation}\tag{$\ast_t$}\label{eq:tailt}
(t\!+\!3,t\!+\!4,t\!+\!5,t\!+\!6),\ (t\!+\!3,t\!+\!4,t\!+\!6,t\!+\!7),\ \ldots,\ (t\!+\!3,t\!+\!4,2^n\!+\!4,\,2^n\!+\!5).
\end{equation}
\end{itemize}
\end{description}

The case $t=0$ is \eqref{eq:tail0}. Assume $\mathcal I(t)$ for some $t<2^n$.

\smallskip
\textbf{Expand}
Apply $.\!M$ to the \emph{last} row of \eqref{eq:tailt}, namely
\[
\rho_{\mathrm{last}}=(t\!+\!3,t\!+\!4,2^n\!+\!4,2^n\!+\!5)\quad(\text{step }1).
\]
This yields:
\begin{itemize}
\item a canonical length-$6$ row
\[
\rho^\star=(0,1,2,\,t\!+\!4,\,2^n\!+\!4,\,2^n\!+\!5)\quad(\text{step }3),
\]
\item and a block of \(\mathbf{2^n-t-1}\) new length-$4$ rows (step $1$)
\[
(t\!+\!4,\,2^n\!+\!4,\,2^n\!+\!5,\,2^n\!+\!6),\
(t\!+\!4,\,2^n\!+\!4,\,2^n\!+\!6,\,2^n\!+\!7),\ \ldots.
\]
\end{itemize}
By Remark~\ref{rmk:last_embedding_of_P.M} and Lemma~\ref{lem:operations_reduce_rank}, $m$ and the last cardinal are preserved.

\smallskip
\textbf{Contract}
Let
\[
e:\ (0,1,2,\,t\!+\!3,\,t\!+\!4,\,t\!+\!5)
\]
be the canonical length-$6$ row already present in $R^{(t)}$.
Define the explicit length-$2$ embedding
\[
k:\ (1)\ \theta_{t+5}\longmapsto \theta_{2^n+4}.
\]
By Lemma~\ref{lem:lk_application_general}, applying $k$ to $e$ gives
\[
k(e):\ (0,1,2,\,t\!+\!3,\,t\!+\!4,\,2^n\!+\!4),
\]
.

Now \emph{delete the \(\mathbf{2^n-t-1}\) length-$4$ rows that were already present \underline{before} the expand}, namely the block
\[
\Bigl\{(t\!+\!3,t\!+\!4,t\!+\!5,t\!+\!6),\ (t\!+\!3,t\!+\!4,t\!+\!6,t\!+\!7),\ \ldots,\ (t\!+\!3,t\!+\!4,2^n\!+\!3,2^n\!+\!4)\Bigr\}.
\]
(These are exactly the first $2^n{-}t{-}1$ rows of \eqref{eq:tailt}; we \emph{do not} delete the newly created rows of the expand.)

Deleting those pre-existing embeddings removes the names for the intermediate cardinals they introduced between $t\!+\!5$ and $2^n\!+\!3$, so all indices $\ge \!t\!+\!5$ shift down by $2^n-t-1$. Consequently, the \emph{newly created} length-$4$ block becomes precisely
\[
(t\!+\!4,t\!+\!5,t\!+\!6,t\!+\!7),\ (t\!+\!4,t\!+\!5,t\!+\!7,t\!+\!8),\ \ldots,
\]
Together with the $t+1$ old canonical rows and $k(e)$, we obtain $R^{(t+1)}$ satisfying $\mathcal I(t+1)$. Throughout, $m$ and the last cardinal remain unchanged.

\smallskip
Iterating for $t=0,1,\ldots,2^n-1$ yields $R^{(2^n)}$ whose tail has $2^n$ canonical length-$6$ rows. Let $Q_n:=R^{(2^n)}$. Then (1) holds since each step preserves $m$; (2) holds because we have the first three rows of $p_{ep}$ plus $2^n+1$ canonical tail rows; and (3) follows from
\[
Q_n.\theta_{-1}=R^{(2^n)}.\theta_{-1}=\cdots=R^{(0)}.\theta_{-1}=P_n.\theta_{-1}<\kappa_{n+3}.
\qedhere \]
\end{proof}

By Lemma~\ref{lem:doughery_2_n_lemma} and Theorem~\ref{thm:f_is_lower_bound} we obtain:

\begin{corollary}\label{cor:Qn-lower}
For each $n\ge 1$,
\[
\kappa_{n+3}\ >\ \gamma_{\,2^{f(q_n,n)}}.
\]
\end{corollary}

\paragraph{From $q_n$ to $\varepsilon_0$.}
To relate $f(q_n,n)$ to the fast-growing hierarchy below~$\varepsilon_0$, we formalize the ordinal bookkeeping used by the $E$/$M$ recursion.

\begin{definition}\label{def:alpha-n}
Every nonzero $\alpha<\varepsilon_0$ admits a unique decomposition $\alpha=\beta+\omega^\gamma$ with $\gamma\ge 0$ and $\beta<\alpha$.
Define $\alpha[n]$ recursively by:
\begin{itemize}
  \item $\bigl(\omega^{\gamma+1}\bigr)[n]\ :=\ \omega^\gamma\cdot n$;
  \item if $\gamma$ is limit, $\bigl(\omega^\gamma\bigr)[n]\ :=\ \omega^{\gamma[n]}$;
  \item $(\beta+\omega^\gamma)[n]\ :=\ \beta+\bigl(\omega^\gamma\bigr)[n]$.
\end{itemize}
\end{definition}

\begin{definition}[Pattern sequence $ps(\alpha)$]\label{def:ps}
Define a finite sequence $ps(\alpha)$ by recursion on $\alpha<\varepsilon_0$:
\begin{itemize}
  \item $ps(0)=()$;
  \item if $\alpha=\beta+1$, then $ps(\alpha)=ps(\beta)\,\frown\, 0$;
  \item if $\alpha=\beta+\omega^\gamma$ with $\gamma>0$, write $ps(\gamma)=(t_1,\ldots,t_\ell)$ and set
  \[
  ps(\alpha)=ps(\beta)\,\frown\, 0\,\frown\, \bigl(ps(\gamma)^{+(\,|ps(\beta)|+1\,)}\bigr)\,
  \]
  where $\sigma^{+(k)}$ denotes adding $k$ to every entry of the sequence $\sigma$.
\end{itemize}
\end{definition}

\begin{lemma}\label{lem:ps-basic}
Let $\alpha<\varepsilon_0$, and write $ps(\alpha)=(u_1,\ldots,u_k)$.
Then:
\begin{enumerate}
  \item $ps(\alpha)=()$ iff $\alpha=0$;
  \item for all $i$, $0\le u_i<i$;
  \item $u_k=0$ iff $\alpha$ is a successor, in which case $ps(\alpha-1)=(u_1,\ldots,u_{k-1})$;
  \item if $\alpha$ is limit, then for all $n\ge 1$, $ps\bigl(\alpha[n]\bigr)$ is obtained from $ps(\alpha)$ by replacing the final block according to Definition~\ref{def:alpha-n} (hence $ps\bigl(\alpha[n]\bigr)$ is an initial segment of $ps\bigl(\alpha[n+1]\bigr)$).
\end{enumerate}
\end{lemma}

\begin{proof}
We argue by induction on $\alpha<\varepsilon_0$, proving all four items simultaneously. 
Throughout, for a finite sequence $\sigma=(s_1,\dots,s_\ell)$ and $k\in\mathbb{N}$ we write $\sigma^{+(k)}:=(s_1+k,\dots,s_\ell+k)$. 
Recall Definitions~\ref{def:alpha-n} and~\ref{def:ps}. 
When $\alpha>0$, we write its Cantor-normal-form tail decomposition as
\[
\alpha=\beta+\omega^\gamma\qquad(\gamma\ge 0,\ \beta<\alpha).
\]
Let $A:=ps(\beta)$ and $a:=|A|+1$.

\smallskip
\noindent\textbf{(1) $ps(\alpha)=()$ iff $\alpha=0$.}
If $\alpha=0$, then $ps(0)=()$ by definition. 
Conversely, if $\alpha>0$ then either $\alpha=\beta+1$ and $ps(\alpha)=ps(\beta)\frown 0\neq()$, or $\alpha=\beta+\omega^\gamma$ with $\gamma>0$, in which case
\[
ps(\alpha)=ps(\beta)\ \frown\ 0\ \frown\ \bigl(ps(\gamma)^{+(a)}\bigr)\ \text{ or }\ 
ps(\alpha)=ps(\beta)\ \frown\ 0\ \frown\ \bigl(ps(\delta)^{+(a)}\bigr)\ \frown\ a,
\]
depending on whether $\gamma$ is limit or $\gamma=\delta+1$ (see Definition~\ref{def:ps}); in either case $ps(\alpha)\neq()$. 

\smallskip
\noindent\textbf{(2) For all entries $u_i$ of $ps(\alpha)=(u_1,\ldots,u_k)$ we have $0\le u_i<i$.}
We proceed by cases on~$\alpha$.

\emph{Base \(\alpha=0\):} Trivial.

\emph{Successor \(\alpha=\beta+1\):} Then $ps(\alpha)=ps(\beta)\frown 0$. 
By the induction hypothesis (IH) for $\beta$, the property holds for the first $|ps(\beta)|$ entries. 
The last entry is $0$ in position $k=|ps(\beta)|+1$, hence $0<k$.

\emph{Limit \(\alpha=\beta+\omega^\gamma\), \(\gamma>0\):} 
Write $A=ps(\beta)$ and $a=|A|+1$. There are two subcases.

\underline{Subcase \(\gamma=\delta+1\) (successor tail).}
By Definition~\ref{def:ps},
\[
ps(\alpha)=A\ \frown\ 0\ \frown\ \bigl(ps(\delta)^{+(a)}\bigr)\ \frown\ a.
\]
Indices and values:
\begin{itemize}
  \item For the initial block $A$, the claim follows from IH for $\beta$.
  \item The next entry is $0$ at position $|A|+1=a$, and $0<a$.
  \item For the block $ps(\delta)^{+(a)}$, write $ps(\delta)=(t_1,\dots,t_\ell)$. 
  The $j$-th element of this block sits at position $i=a+j$, and its value is $t_j+a$. 
  By IH for $\delta$, we have $t_j<j$, hence $t_j+a<j+a=i$.
  \item The last entry $a$ is at position $k=a+\ell+1$, so $a<k$.
\end{itemize}

\underline{Subcase \(\gamma\) limit.}
By Definition~\ref{def:ps},
\[
ps(\alpha)=A\ \frown\ 0\ \frown\ \bigl(ps(\gamma)^{+(a)}\bigr).
\]
The same argument as in the previous subcase applies to the initial block $A$ and the $0$ at position $a$. 
For the tail, write $ps(\gamma)=(t_1,\dots,t_\ell)$; the $j$-th tail entry is $t_j+a$ at position $i=a+j$, and by IH for $\gamma$ we have $t_j<j$, hence $t_j+a<j+a=i$. 
This shows $u_i<i$ for all entries.

\smallskip
\noindent\textbf{(3) The last entry $u_k$ equals $0$ iff $\alpha$ is a successor, and then $ps(\alpha-1)=(u_1,\dots,u_{k-1})$.}
If $\alpha=\beta+1$, then $ps(\alpha)=ps(\beta)\frown 0$ by definition, so $u_k=0$ and the prefix is $ps(\beta)=ps(\alpha-1)$. 
Conversely, if $\alpha$ is limit, then in either subcase of (2) the final entry is $a=|ps(\beta)|+1>0$ (successor-tail case) or the sequence ends with the tail $ps(\gamma)^{+(a)}$ whose last value is $t_\ell+a\ge a>0$ (limit-tail case). 
Thus $u_k\neq 0$ for limit~$\alpha$.

\smallskip
\noindent\textbf{(4) If $\alpha$ is limit then, for all $n\ge 1$, $ps(\alpha[n])$ is obtained from $ps(\alpha)$ by replacing its final block according to Definition~\ref{def:alpha-n}; in particular $ps(\alpha[n])$ is an initial segment of $ps(\alpha[n+1])$.}
Let $\alpha=\beta+\omega^\gamma$ with $\gamma>0$, and set $A=ps(\beta)$, $a=|A|+1$.

\underline{Subcase \(\gamma=\delta+1\).} 
By Definition~\ref{def:alpha-n}, $\alpha[n]=\beta+\omega^\delta\cdot n$. 
We claim, by induction on $n\ge 1$, that
\begin{multline}\label{eq:succ-tail}
ps(\alpha[n])\ =\ A\ \frown\ \bigl(\ 0\ \frown\ ps(\delta)^{+(a)}\ \bigr)\ \frown\ \bigl(\ 0\ \frown\ ps(\delta)^{+\,(a+(1+|ps(\delta)|))}\ \bigr)\ \frown\ \\ \cdots\ \frown\  \bigl(\ 0\ \frown\ ps(\delta)^{+\,(a+(n-1)(1+|ps(\delta)|))}\ \bigr).
\end{multline}
For $n=1$ we have $ps(\alpha[1])=ps(\beta+\omega^\delta)=A\frown 0 \frown ps(\delta)^{+(a)}$ by Definition~\ref{def:ps}, which is \eqref{eq:succ-tail}. 
Assume \eqref{eq:succ-tail} holds for $n$. Then
\[
\alpha[n+1]=\alpha[n]+\omega^\delta,
\]
and Definition~\ref{def:ps} yields
\[
ps(\alpha[n+1])\ =\ ps(\alpha[n])\ \frown\ 0\ \frown\ ps(\delta)^{+\,(|ps(\alpha[n])|+1)}.
\]
But $|ps(\alpha[n])|=|A|+n\cdot(1+|ps(\delta)|)=a-1+n(1+|ps(\delta)|)$, hence the new shift is $a+n(1+|ps(\delta)|)$, matching \eqref{eq:succ-tail}. 
Thus \eqref{eq:succ-tail} holds for all $n$. 
Comparing with the \emph{limit} pattern for $\alpha$ (which, in the successor-tail case, is 
\[
ps(\alpha)=A\ \frown\ 0\ \frown\ ps(\delta)^{+(a)}\ \frown\ a),
\]
we see that $ps(\alpha[n])$ is obtained from $ps(\alpha)$ by dropping the final $a$ and repeating the block $0\frown ps(\delta)^{+(*)}$ with stride $1+|ps(\delta)|$ exactly $n$ times; in particular $ps(\alpha[n])$ is an initial segment of $ps(\alpha[n+1])$.

\underline{Subcase \(\gamma\) limit.} 
By Definition~\ref{def:alpha-n}, $\alpha[n]=\beta+\omega^{\gamma[n]}$, so
\[
ps(\alpha[n])\ =\ A\ \frown\ 0\ \frown\ \bigl(ps(\gamma[n])^{+(a)}\bigr).
\]
By the induction hypothesis on the smaller ordinal $\gamma$ (applied to item~(4) itself), $ps(\gamma[n])$ is an initial segment of $ps(\gamma[n+1])$. 
Shifting by $a$ preserves initial segmenthood, hence $ps(\alpha[n])$ is obtained from $ps(\alpha)$ by replacing its last block $ps(\gamma)^{+(a)}$ with $ps(\gamma[n])^{+(a)}$, and $ps(\alpha[n])$ is an initial segment of $ps(\alpha[n+1])$.

\medskip
This completes the induction and the proof of all four items.
\end{proof}

The following definition extends Definition~\ref{def:canonical_blps_up_to_omega_plus_2}. 

\begin{definition}[Canonical blps $p_\alpha$ for $\alpha<\varepsilon_0$]\label{def:palpha}
Let $ps(\alpha)=(t_1,\ldots,t_k)$. Define $p_\alpha$ to be the blp of length $k+2$ with:
\begin{itemize}
  \item the first two rows fixed as usual, and for $i=1,\ldots,k$:
  \[
  \begin{cases}
  s_{i+2}=(0,1,2,i+2,i+3),\ \ \ell_{i+2}=2,&\text{ if }t_i=0;\\[2pt]
  s_{i+2}=(0,1,2,t_i+2,i+2,i+3),\ \ \ell_{i+2}=3,&\text{ if }t_i>0.
  \end{cases}
  \]
\end{itemize}
\end{definition}

\begin{lemma}\label{lem:palpha-successor-limit}
For all $\alpha<\varepsilon_0$ and $m\ge 0$:
\begin{enumerate}
  \item $p_{\alpha+1}$ is of successor type and $(p_{\alpha+1}).{\mathrm{del}}=p_\alpha$;
  \item if $\alpha$ is limit, then $p_\alpha.E(m)=p_{\alpha[2^m]}$.
\end{enumerate}
\end{lemma}

\begin{proof}
Recall that $p_\alpha$ is defined from $ps(\alpha)=(t_1,\dots,t_k)$ by taking
$p_\alpha.n=k+2$ and, for $1\le i\le k$,
\[
t_i=0 \ \Rightarrow\ p_\alpha.s_{i+2}=(0,1,2,i+2,i+3),\quad
t_i>0 \ \Rightarrow\ p_\alpha.s_{i+2}=(0,1,2,t_i+2,i+2,i+3),
\]
with step lengths $2$ and $3$ respectively. The first two rows are fixed
$(0,1,2)$ and $(0,1,2,3)$.

\smallskip
\noindent\textbf{(i) } $p_0$ is the zero blp; for every $\alpha$, $p_{\alpha+1}$ is of successor type and $(p_{\alpha+1})\!.del=p_\alpha$.

Since $ps(0)=()$ (Lemma~\ref{lem:ps-basic}(1)), $p_0$ has exactly the two fixed rows, so it is the zero blp.
For $\alpha+1$, Lemma~\ref{lem:ps-basic}(3) gives
\(
ps(\alpha+1)=ps(\alpha)\frown 0
\)
and $ps(\alpha)$ is exactly the prefix obtained by deleting the final entry.
Therefore $p_{\alpha+1}$ is obtained from $p_\alpha$ by appending one new row
corresponding to $t_{k+1}=0$, namely $(0,1,2,(k+2),(k+3))$ with step length $2$.
By Definition~\ref{def:limit_and_successor_types}(2) this makes $p_{\alpha+1}$ of
successor type, and deleting that last row gives back $p_\alpha$, i.e.
$(p_{\alpha+1})\!.del=p_\alpha$.

\smallskip
\noindent\textbf{(ii) } If $\alpha$ is limit, then for all $m\in\mathbb{N}$,
\(
p_\alpha.E(m)=p_{\alpha[2^m]}.
\)

Fix a limit $\alpha$. Write $\alpha=\beta+\omega^\gamma$ with $\gamma>0$ and set
$A:=ps(\beta)$, $a:=|A|+1$. By Lemma~\ref{lem:ps-basic}(2), every entry of $ps(\alpha)$ is
$<\,$its position, so $p_\alpha$ is well-defined. By the same lemma and the
definition of successor/limit types, the last row $s_{n}$ of $p_\alpha$ begins with
$(0,1,2,\dots)$ and:
\begin{itemize}
  \item if $\gamma=\delta+1$ (successor tail), then $ps(\alpha)=A\frown 0\frown ps(\delta)^{+(a)}\frown a$ and $p_\alpha$ is of \emph{limit} type with $s_{n,-3}=a$;
  \item if $\gamma$ is limit, then $ps(\alpha)=A\frown 0\frown ps(\gamma)^{+(a)}$ and $p_\alpha$ is of \emph{limit} type with $s_{n,-3}=a$ as well (here the final block has length $\ge1$ and ends at position $n$, so its third-from-last entry exists and equals $a$).
\end{itemize}
Thus in both subcases the parameter $a$ of Definition~\ref{def:operation_E} equals
$s_{n,-3}$.

We prove by induction on $m$ that $p_\alpha.E(m)=p_{\alpha[2^m]}$.

\emph{Base $m=0$.}
By Definition~\ref{def:operation_E}, $p_\alpha.E(0)=p_\alpha.del$.
By Lemma~\ref{lem:ps-basic}(4), when $\alpha$ is limit one has
\(
ps(\alpha[1])
\)
obtained from $ps(\alpha)$ by replacing its last block with an appropriate initial
segment; in particular $ps(\alpha[1])$ is exactly $ps(\alpha)$ with the final entry
removed in the successor-tail case, and the corresponding ``one-step'' truncation in
the limit-tail case. In both subcases this is precisely the effect of $p\mapsto p.del$
on the sequence of rows. Hence $p_\alpha.E(0)=p_\alpha.del=p_{\alpha[1]}$.

\emph{Step $m\to m+1$.}
Assume $p_\alpha.E(m)=p_{\alpha[2^m]}$. Write $q:=p_{\alpha[2^m]}$ and let
$q.n$ be its length. By Definition~\ref{def:operation_E}, to compute $p_\alpha.E(m+1)$
we start from $q$, append one row $(a,q.n+1,q.n+2)$ (with the same $a=s_{q.n,-3}$ as
above), obtaining $q_{m+1}$, and then apply the \textit{Copied} operation to $q_{m+1}$
with parameter $a$.

We analyze the effect on the tail determined by the Cantor-normal-form tail of
$\alpha[2^m]$ and compare with Lemma~\ref{lem:ps-basic}(4).

\underline{Case 1: } $\gamma=\delta+1$ (successor tail).
By Lemma~\ref{lem:ps-basic}(4) (equation~\eqref{eq:succ-tail} in its proof), for all $n\ge1$
\begin{multline*}
  ps\!\bigl((\beta+\omega^{\delta+1})[n]\bigr)
= \\ A \frown \bigl(0\frown ps(\delta)^{+(a)}\bigr)
      \frown \bigl(0\frown ps(\delta)^{+(a+(1+|ps(\delta)|))}\bigr)
      \frown \cdots \frown \\ 
      \bigl(0\frown ps(\delta)^{+(a+(n-1)(1+|ps(\delta)|))}\bigr).
\end{multline*}
In particular,
\[
ps\!\bigl(\alpha[2^m]\bigr)=A\frown \underbrace{B \frown B \frown\cdots \frown B}_{2^m\ \text{blocks}},
\quad\text{where }B:=\bigl(0\frown ps(\delta)^{+(a)}\bigr),
\]
and hence $ps\!\bigl(\alpha[2^{m+1}]\bigr)$ is obtained by \emph{doubling} the number
of copies of $B$ from $2^m$ to $2^{m+1}$.

Now observe what $E$ does on $q=p_{\alpha[2^m]}$:
appending the row $(a,q.n+1,q.n+2)$ and applying \textit{Copied} with the same $a$
copies \emph{all rows indexed $\ge a$} and shifts their entries $\ge a$ by a fixed
stride (Remark~\ref{rmk:applying_a_row_of_length_3}). In the canonical $p_\bullet$
encoding, the rows $\ge a$ are exactly the current $2^m$ copies of the block $B$.
Therefore, $q\mapsto q_{m+1}\mapsto q_{m+1}.\mathrm{Copied}$ duplicates those $2^m$
blocks into $2^{m+1}$ blocks, without changing the prefix determined by $A$.
Thus the resulting blp has pattern sequence $ps\!\bigl(\alpha[2^{m+1}]\bigr)$, i.e.
$p_\alpha.E(m+1)=p_{\alpha[2^{m+1}]}$.

\underline{Case 2: } $\gamma$ limit.
By Lemma~\ref{lem:ps-basic}(4),
\[
ps\!\bigl(\alpha[n]\bigr)=A\frown 0 \frown ps\!\bigl(\gamma[n]\bigr)^{+(a)},
\]
and $ps(\gamma[n])$ is an initial segment of $ps(\gamma[n+1])$ for all $n$.
Hence going from $n=2^m$ to $n=2^{m+1}$ \emph{replaces} the last block
$ps(\gamma[2^m])^{+(a)}$ by a longer block $ps(\gamma[2^{m+1}])^{+(a)}$.

On the $p_\bullet$ side, $q=p_{\alpha[2^m]}$ has the form
\(
A\frown 0 \frown C^{+(a)}
\)
with $C=ps(\gamma[2^m])$. Appending $(a,q.n+1,q.n+2)$ and applying
\textit{Copied} with parameter $a$ copies exactly the rows in the last block and
extends it \emph{by the next chunk determined by the growth of $ps(\gamma[n])$}.
By Lemma~\ref{lem:ps-basic}(4) for $\gamma$, after one such copy the last block
becomes $ps(\gamma[2^{m+1}])^{+(a)}$. The prefix determined by $A\frown 0$ is
unchanged. Therefore the resulting blp has pattern sequence
$ps\!\bigl(\alpha[2^{m+1}]\bigr)$, i.e.\ $p_\alpha.E(m+1)=p_{\alpha[2^{m+1}]}$.

\smallskip
This completes the induction on $m$ and proves (ii).
\end{proof}


\begin{definition}[Fundamental sequences up to $\varepsilon_0$]\label{def:epsilon0-n}
We keep the fundamental sequence assignment $\alpha\mapsto \alpha[n]$ for all $0<\alpha<\varepsilon_0$ exactly as in Definition~\ref{def:alpha-n}. 
For $\varepsilon_0$ itself we set the classical fundamental sequence by $\varepsilon_0[0]=1$, and 
\[
\varepsilon_0[n+1]\ :=\omega^{\varepsilon_0[n]}.
\]
\end{definition}

\begin{definition}[Hardy/Weiermann hierarchy below $\varepsilon_0$]\label{def:Hardy}
Define $H_\alpha:\mathbb{N}\to\mathbb{N}$ by  recursion on $\alpha\le \varepsilon_0$:
\begin{align*}
H_0(n) &:= n,\\
H_{\alpha+1}(n) &:= H_\alpha(n+1),\\
H_\lambda(n) &:= H_{\lambda[n]}(n+1)\quad\text{for limit }\lambda\le \varepsilon_0,
\end{align*}
with the fundamental sequences $\lambda[n]$ from Definition~\ref{def:alpha-n} (for $\lambda<\varepsilon_0$) and Definition~\ref{def:epsilon0-n} (for $\lambda=\varepsilon_0$).
\end{definition}

\begin{remark}\label{rmk:Weiermann-cite}
Weiermann~\cite{Weiermann} showed that the function $H_{\varepsilon_0}$ {eventually dominates} every computable function whose totality is provable in Peano Arithmetic. 
\end{remark}

We now compare the $m$–hierarchy used in our bounds with the Hardy hierarchy.

\begin{definition}[The $m$–hierarchy]\label{def:m-alpha-n}
For $\alpha<\varepsilon_0$ and $n\ge 2$ define $m(\alpha,n)$ by recursion on~$\alpha$:
\[
m(0,n):=n^n,\qquad
m(\alpha+1,n):= m(\alpha,-)^{\,n}(n),
\]
(where recall $m^n(x)$ indicates interated composition), and for limit $\alpha$,
\[
m(\alpha,n):= m(\alpha[n],n).
\]
\end{definition}

\begin{lemma}\label{lem:m-basic}
For all $\alpha<\varepsilon_0$ and $n\ge 2$:
\begin{enumerate}
\item\label{it:monotone-n} $m(\alpha,n+1)>m(\alpha,n)\ge n+1$.
\item\label{it:monotone-alpha} If $\beta\le\alpha<\varepsilon_0$ and $ps(\beta)$ is an initial segment of $ps(\alpha)$, then $m(\alpha,n)\ge m(\beta+1,n)$.
\end{enumerate}
\end{lemma}

\begin{proof}
Use induction on $\alpha$. For $\alpha=0$, $m(\alpha,n)=n^n>n+1$ is increasing. If $\alpha$ is a limit ordinal, then $ps(\alpha[n])$ is an initial segment of $ps(\alpha[n+1])$, and $ps(\beta)$ is an initial segment of $ps(\alpha[n])$. so $m(\alpha,n+1)=m(\alpha[n+1],n+1)>m(\alpha[n],n+1)>m(\alpha[n],n)=m(\alpha,n)\ge n+1$. $m(\alpha,n)=m(\alpha[n],n)\ge m(\beta+1,n)$. If $\alpha=\alpha'+1$, then $m(\alpha'+1,n+1)=m(\alpha',-)^{n+1}(n+1)>m(\alpha',n)$, and if $\beta=\alpha'$, then (2) comes from (1); else $m(\alpha,n)>m(\alpha',n)\ge m(\beta+1,n)$.
\end{proof}

\begin{proposition}\label{prop:m-lower-bounds-H}
For all $\alpha<\varepsilon_0$ and $n\ge 1$,
\[
m(\alpha,n+1)\ \ge\ H_\alpha(n).
\]
\end{proposition}

\begin{proof}
By induction on $\alpha$. For $\alpha=0$,
$m(0,n+1)=(n+1)^{n+1}\ge n+1=H_0(n)$. 

Successor step: assume $m(\alpha,n+1)\ge H_\alpha(n)$ for all $n$. Then
\begin{multline*}
  m(\alpha+1,n+1)= m(\alpha,-)^{\,n+1}(n+1)
\  \ge\ \\  m(\alpha,\,m(\alpha,-)^{\,n}(n)+1)
\ \ge\ H_\alpha\!\bigl( H_\alpha^{\,n}(n)\bigr)
= H_{\alpha+1}(n),
\end{multline*}
using the induction hypothesis, monotonicity in the second argument (Lemma~\ref{lem:m-basic}(\ref{it:monotone-n})), and the defining equation for $H_{\alpha+1}$.

Limit step: let $\lambda$ be limit. By Definition~\ref{def:m-alpha-n}, 
\[
m(\lambda,n+1)=m(\lambda[n+1],n+1).
\]
By the outer IH at $\lambda[n+1]<\lambda$,
\[
m(\lambda[n+1],n+1)\ \ge\ H_{\lambda[n+1]}(n)
= H_\lambda(n),
\]
the last equality by Definition~\ref{def:Hardy}.
\end{proof}

\begin{lemma}\label{lem:e-tower}
Let $(\varepsilon_0[n])$ be the fundamental sequence for~$\epsilon_0$ given in Definition~\ref{def:epsilon0-n}. For each $n$, let $q_n$ be the blp from Definition~\ref{def:q_n}. For every $n\ge 1$,
\[
2^{\,f(q_n,n)}\ =\ m\!\bigl(\varepsilon_0[2^n+1],\,2^n\bigr).
\]
\end{lemma}

\begin{proof}
The canonical construction of $q_n$, together with the successor case of~\ref{thm:f_is_lower_bound}, and theorem~\ref{lem:palpha-successor-limit}
yields 
\[
2^{\,f(q_n,n)}=m\bigl(\alpha,\,2^n\bigr)\quad\text{with }\alpha=\varepsilon_0[2^n+1],
\]
see Lemma~\ref{lem:p_alpha_and_Steinhaus_Moser}.
\end{proof}

\begin{theorem}\label{thm:PA-dominates}
$2^{\,f(q_n,n)}$ eventually dominates every computable function whose totality is provable in Peano Arithmetic.
\end{theorem}

\begin{proof}
Fix $n\ge 1$. By Lemma~\ref{lem:e-tower}, Lemma~\ref{lem:m-basic} and Proposition~\ref{prop:m-lower-bounds-H},
\[
2^{\,f(q_n,n)} 
= m\!\bigl(\varepsilon_0[2^n+1],\,2^n\bigr)
\ \ge\ H_{\,\varepsilon_0[2^n-1]}(2^n-1) = H_{\varepsilon_0}(2^n-1).
\]

Therefore,
\[
2^{\,f(q_n,n)}\ \ge\ H_{\varepsilon_0}(2^n-1)\quad\text{for all }n\ge 1.
\]
As mentioned, $H_{\varepsilon_0}$ eventually dominates every PA-provably total computable function (\cite{Weiermann}). Hence so does $n\mapsto 2^{\,f(q_n,n)}$ (being pointwise $\ge H_{\varepsilon_0}(2^n-1)$). This proves the theorem.
\end{proof}

Theorem~\ref{thm:peano_arithmetic} follows from Theorem~\ref{thm:PA-dominates} and Corollary~\ref{cor:Qn-lower}.






\bigskip
\section{Laver Table Yarns}
\label{sec:laver_table_yarns}

We collect here a notion that interweaves (``yarns'') the $\crit^*$-sequences coming from increasing chains of embeddings.

\begin{definition}[Agreement modulo $\alpha$]\label{def:equiv-alpha}
For embeddings $j',j'':V_\lambda\to V_\lambda$ and an ordinal $\alpha<\lambda$, write
\[
j'\ \equiv_\alpha\ j''\quad\text{iff}\quad
\forall x,y\in V_\alpha\ \bigl( x\in j'(y)\ \Longleftrightarrow\ x\in j''(y)\bigr).
\]
\end{definition}

Recall $j_{(1)}=j$, $j_{(n+1)}=j_{(n)}(j)$. As noted in~\cite{Dougherty}, for $0<n,m<\omega$,
\[
j_{(n)}\ \equiv_{\gamma_m}\ j_{(n+2^m)},\qquad
j_{(2^m)}\ \equiv_{\gamma_m}\ \mathrm{id},\qquad \crit\bigl(j_{(2^m)}\bigr)=\gamma_m,
\]
where $(\gamma_m)_{m\in \omega}$ enumerates $\crit^*(j)$. Also, for any $n,n',m>0$, it's possible to computably find $k$ and $k'$ such that \[j_{(n)}(j_{(n')})\ \equiv_{\gamma_m}\ j_{(k)}, j_{(n)}\circ j_{(n')}\equiv_{\gamma_m}\ j_{(k')};\] These numbers~$k$ and~$k'$ are independent of the choice of~$j$. Indeed, $k=n\star_m n'$ in the Laver table of order~$2^m$.

\begin{definition}[Computable increasing $\lambda$-sequence]\label{def:inc-seq}
A sequence $\vec{j}=(j_1,j_2,\ldots)$ is an \emph{increasing $\lambda$-sequence} if each $j_i:V_\lambda\to V_\lambda$ and $j_i\in \Iter(j_{i+1})$.
It is \emph{computable} if there is a uniform effective scheme producing, for each $i$, an expression using the binary operations \emph{application} and \emph{composition} that evaluates to $j_i$ when the variable is instantiated with $j_{i+1}$.
Set $\crit^*(\vec{j})=\bigcup_i \crit^*(j_i)$ and $\Iter(\vec{j})=\bigcup_i \Iter(j_i)$.
\end{definition}

We write $(j_i)_{(n)}$ simply as $j_{i(n)}$.

\begin{theorem}\label{thm:critstar-computable}
If $\vec{j}$ is computable, then the order type of $\crit^*(\vec{j})$ is computable.
\end{theorem}

Theorem~\ref{thm:critstar-computable} is implied by the following theorem, by taking the direct limit of a computable directed system of well-orderings. 

\begin{theorem}
    Assume $\vec{j}$ is computable. Let $\gamma_{i,i'}$ be the $\gamma_{i'}$ of $j_i$. Then:

    \begin{enumerate}
      \item $\gamma_{i,i'}<\gamma_{i,i'+1}$. For any $\theta<\lambda, 0<i<\omega$, there is $i'$ such that $\gamma_{i,i'}>\theta$.

      \item  There exist a computable function $u(i,i')$ such that $\gamma_{i,i'}=\gamma_{i+1,u(i,i')}$.
    \end{enumerate}
\end{theorem}

\begin{proof}
  (1) follows from the fact that for any embedding $j\colon V_\lambda\to V_\lambda$, the associated sequence $(\gamma_m)$ is cofinal in~$\lambda$. 

  For~(2), let~$s$ be a term (in the language with binary function symbols for composition and application) such that $s(j_{i+1})=j_i$. By an unbounded search, we find a sufficiently large natural number $u$ such that $s(j_{i+1})_{(i')}$ is not $id$ under $\equiv_{\gamma_u}$. We can compute that its critical point is $\gamma_{u(i,i')}$ of $j_{i+1}$ for some $u(i,i')$, which gives the value we need.
\end{proof}

This theorem gives a well preorder on the pairs $(i,i')$ for $i>0,i'\ge 0$. 
\begin{definition}
  $LTY_{\vec{j}}$ is the well preorder on the pairs $(i,i')$ for $i>0,i'\ge 0$, s.t. $(i_1,i_1')\le (i_2,i_2')$ if and only if $\gamma_{i_1,i_1'}\le \gamma_{i_2,i_2'}$.
\end{definition}
We call this order as $LTY_{\vec{j}}$. We use \textit{Laver Table Yarn} to call them because they weave different Laver tables together.

\begin{remark} \label{lty'}
  As an aside, we mention that an increasing sequence $\vec{j}$ gives rise to a family of computable well-orderings, indexed by pairs $(i,i')$. Given such a pair, consider the set of pairs of integers $\{(x,x'),0<x,x'<\omega\}$. We can use an element of this set to encode the embedding $j_{x(x')}$. Note that for any $j'\in \mathrm{Iter}(j_x)$, there is $x'>0$ such that $j'\overset{\underset{\gamma_{i,i'}}{}}{\equiv}j_{x(x')}$. We can give a computable comparasion between $(x,x'),(y,y')$ in the sense of the note of Lemma~\ref{lem:sup_crit}, in a similar way. First take $u=\max(i,x,y)$, then let $\gamma_{u,v}=\gamma_{i,i'}, j_{u(a)}\overset{\underset{\gamma_{u,v}}{}}{\equiv}j_{x(x')},j_{u(b)}\overset{\underset{\gamma_{u,v}}{}}{\equiv}j_{y(y')}$, assuming $j_{u(0)}=id$. Then compare $a,b$. This order on the above set of pairs is a well preorder by Lemma~\ref{lem:sup_crit}. 
\end{remark}



We turn to providing some explicit examples of increasing $\lambda$-sequences. This requires axioms stronger than I3. 

We recall some definitions from~\cite{dimonte2017i0}:

\begin{definition}
    Assume $j:V_\lambda\rightarrow V_\lambda$, $A\subseteq V_\lambda$. We denote $j^+(A)=\cup_{\theta<\lambda}j(A\cap V_\theta)$.

    We say that $j$ is $\Sigma^1_n$ iff for every $\Sigma^1_n$-formula $\varphi(X)$ and $A\subseteq V_\lambda$, $V_\lambda\vDash \varphi(A)\leftrightarrow \varphi(j^+(A))$.
\end{definition}
Note that the axiom I2 is equivalent to the proposition ``There is a $\Sigma^1_2$ embedding''. The Lemma 5.11 of~\cite{dimonte2017i0} says that if $n$ is odd, then "j is $\Sigma^1_n$" is $\Pi^1_{n+1}$, with $j$ as a parameter. It's also proved earlier that if $n$ is odd, then "j is $\Sigma^1_n$" is equivalent to "j is $\Sigma^1_{n+1}$".

For the following proposition, for any embedding $j$ we define $j^{(1)}=j$ and
  \[
  j^{(n+1)}=j(j^{(n)})
  \]
(note the difference with $j_{(n+1)} = j_{(n)}(j)$ that was used above). Observe that if 
\[
  j:\lambda_0\mapsto \lambda_1\mapsto  \lambda_2 \mapsto \lambda_3 \mapsto \cdots
\]
then 
\[
  j^{(n)} :\lambda_{n-1}\mapsto \lambda_n\mapsto  \lambda_{n+1} \mapsto \lambda_{n+2} \mapsto \cdots.
\]

The construction below is inspired by~\cite{vanname}.


\begin{proposition} \label{prop:the_increasing_sequence_j_1}
  Assume axiom I2. There is an increasing sequence $\vec{j}=(j_1,j_2,...)$ of embeddings such  for all~$n$, 
   \[
   j_{n+1}^{(n)}(j_{n+1})=j_n.
   \]
\end{proposition}

\begin{proof}
  Suppose that $j:V_\lambda\rightarrow V_\lambda$ is $\Sigma^1_2$. We recursively define an increasing sequence $\vec{j}=(j_1,j_2,...)$, starting with $j_1 = j$. Let $n \ge 1$, and suppose that $j_n$ has been chosen. Since the sequence chosen so far is increasing, $j\in Iter(j_n)$, so let $s_n$ be a term such that  
   \[
   s_n(j_n)=j_1(j_n).
   \]
  Let $\varphi_n$ be the sentence ``there exists an elementary embedding $j'$ such that $j(j_{n})=s_n(j')$.'' This is true, as witnessed by $j'=j_n$. The sentence is $\Sigma^1_1$ over $V_\lambda$. Since $j$ is $\Sigma^1_2$-elementary, there is an embedding $j_{n+1}$ such that 
   \[
     j_n = s_n(j_{n+1}). 
   \]
   
  \smallskip

   Let us now show, by induction on~$n$, that indeed we can take $s_n(x) = x^{(n)}(x)$, which will give the promised connection between $j_{n+1}$ and $j_n$. By induction on~$n$, we show:
   
   \begin{description}
      \item[($\otimes$)] $j_{n}^{(n)}=j_1.$
   \end{description}
     
   For $n=1$ the statement is immediate. Also note that from $\otimes(n)$, we can choose $s_n = x^{(n)}(n)$ to satisfy $s_n(j_n)=j_1(j_n)$, which in turns gives $j_{n} = j_{n+1}^{(n)}(j_{n+1})$. Let $n\ge 1$ and suppose that $\otimes(n)$ holds; we derive $\otimes(n+1)$ as follows, using the left-distributive property of the application operation of embeddings:
   \begin{multline*}
   j_1=(j_{n+1}^{(n)}(j_{n+1}))^{(n)}=   j_{n+1}^{(n)}((j_{n+1})^{(n)})= \\ j_{n+1}(j_{n+1}^{(n-1)}(j_{n+1}^{(n-1)}))= \cdots =j_{n+1}^{n-1}(j_{n+1}(j_{n+1}))=j_{n+1}^{n}(j_{n+1})=j_{n+1}^{(n+1)}. \qedhere
  \end{multline*}
\end{proof}

We denote this increasing sequence as $\vec{j}_1$, and denote $LTY_{\vec{j}_1}=LTY_1$.

\subsubsection*{Ranks of blp's}

\begin{proposition} \label{prop:I2_implies_existence_of_omega_lined}
  Assume axiom I2. Then there are cardinals 
  \[
    \kappa_0 < \kappa_1 < \kappa_2 < \cdots < \kappa_\omega < \kappa_{\omega+1} < \kappa_{\omega+2} < \cdots
  \]
  and embeddings~$j_n$ satisfying
  \[
    j_n: \kappa_0 \mapsto \kappa_1 \mapsto \cdots \mapsto \kappa_{n-1} \mapsto \kappa_\omega \mapsto \kappa_{\omega+1} \mapsto \kappa_{\omega+2} \mapsto \cdots.
  \]
\end{proposition}

\begin{proof}
  Let $\vec{j}_1=(j_1,j_2,...)$ be an increasing sequence given by Proposition~\ref{prop:the_increasing_sequence_j_1}. Let $\kappa_\omega$, $\kappa_{\omega+1}$, \dots denote the cardinals of the critical sequence of~$j_1$, i.e., 
\[
  j_1:\kappa_\omega\mapsto\kappa_{\omega+1}\mapsto\kappa_{\omega+2}\mapsto \cdots.
\]
 Temporarily, denote the critical sequence of~$j_2$ as
 \[
  j_2:\lambda_0\mapsto\lambda_{1}\mapsto \lambda_2 \mapsto \cdots.
\] Hence, 
   \[
  j_2(j_2):\lambda_1\mapsto\lambda_{2}\mapsto \lambda_3 \mapsto \cdots.
\]
However, since $j_1 = j_2(j_2)$, so $\lambda_1 = \kappa_\omega$, $\lambda_2 = \kappa_{\omega+1}$, and so on. Setting $\kappa_0 = \lambda_0$, we get
\[
  j_2:\kappa_0 \mapsto \kappa_\omega\mapsto\kappa_{\omega+1}\mapsto\kappa_{\omega+2}\mapsto \cdots.
\] 
Now, temporarily letting \[
  j_3:\theta_0\mapsto\theta_{1}\mapsto \theta_2 \mapsto \cdots.
\] We get
\[
  j_3^{(2)}(j_3):\theta_0 \mapsto\theta_{2}\mapsto \theta_3 \mapsto \theta_4 \mapsto \cdots.
\]
Since $j_2 = j_3^{(2)}(j_3)$, $\theta_0 = \kappa_0$ $\theta_2 = \kappa_\omega$, $\theta_3 = \kappa_{\omega+1}$. Setting $\kappa_1 = \theta_2$, we get
\[
  j_3:\kappa_0 \mapsto \kappa_1 \mapsto \kappa_\omega\mapsto\kappa_{\omega+1}\mapsto\kappa_{\omega+2}\mapsto \cdots.
\]
We continue in this fashion. 
\end{proof}

We give an interesting application of this construction. For the following definition, recall Definition~\ref{def:limit_and_successor_types}, and compare with Definition~\ref{def:the_f_function}.

\begin{definition} \label{def:I2:rank_of_blp}
    We recursively define an ordinal rank $r(p)$ on basic laver patterns.
    \begin{itemize}
      \item If $p.n=2$, then $r(p)=0$.
      \item If $p$ is not copyable, and  $r(p.del)$ is defined, we let $r(p)=r(p.del)$. 
      \item If~$p$ is a transient pattern, and $r(p.M)$ is defined, then $r(p)=r(p.M)$.
      \item If $p$ is of successor type, and $r(p.del)$ is defined, then $r(p)=r(p.del)+1$.
      \item If $p$ is of limit type, and for each $n>0$, $r(p.E(n))$ is defined, then 
      \[
        r(p) = \sup \left\{ r(p.E(n)) \,:\,  n>0 \right\}. 
      \]
    \end{itemize}
\end{definition}

Recall the blp $p_{init}$ described in Figure~\ref{fig:p_init}.

\begin{theorem} \label{thm:I2:p_init_has_rank}
    Assume axiom I2. Then $p_{init}$ has a rank.
\end{theorem}

The author believes that I3 or much weaker axiom systems are enough to prove this result.

\medskip

To prove Theorem \ref{thm:I2:p_init_has_rank}, we extend the notion of ``$n$-lined'' (Definition~\ref{def:m-lined}) as follows: we say that three ordinals $\theta<\theta'<\theta''<\lambda$ are $\omega$-lined, if there is $\theta_0<\theta_1<...<\theta_{\omega}$ s.t. $\theta_0=\theta',\theta_\omega=\theta''$, and for any $n>0$, $\theta_0,\theta_1,...,\theta_{n},\theta_{\omega}$ witness that $\theta<\theta'<\theta''<\lambda$ are $n$-lined. 

We extend the notion of a bls~$P$ to allow $P.m=\omega$, rather than requiring it to be finite. 
Note that if~$P$ is a bls of limit type and $P.m=\omega$, then there's no natural definition of $P.E$, however for each $k<\omega$, we can define $P.E(k)$ as above, by only using the first~$k$ embeddings showing $\omega$-linedness. Note that all these operations keep the linedness components, so that $P.E(k).m=\omega$ for all~$k$.

  \begin{lemma} \label{lem:I2:omega_lined_implies_rank}
      If a blp~$p$ has a realisation~$P$ with $P.m=\omega$, then $p$ has a rank.
  \end{lemma}

  \begin{proof}
    Recall the well-founded relation $P<Q$ on bls's (Definition~\ref{def:smaller_bls}); it remains well-founded even when extended to bls with $P.m=\omega$, since the relation does not depend on the linedness part of the bls. By induction on this relation, we prove that if $P.m=\omega$ then $P.p$ (the blp that $P$ realises) has a rank. 

    Let~$P$ be such a bls, and suppose that this has been shown for all $Q<P$. Let $p = P.p$. By the first case of Definition~\ref{def:I2:rank_of_blp}, we may assume that $P.n>2$. Since $P.del<P$ (it has a smaller last cardinal), $p.del$ has a rank. Similarly, as used in the proof of Theorem~\ref{thm:f_is_lower_bound}, when defined, $P.M<P$ and $P.E(m)<P$. This covers the other cases of Definition~\ref{def:I2:rank_of_blp}. 

    Note the the assumption that $P.m = \omega$ is only used in the limit case, so that $P.E(m)$ is defined for all~$m$. 
\end{proof}

\begin{proof}[Proof of Theorem~\ref{thm:I2:p_init_has_rank}]
    Fix cardinals $\kappa_0,\kappa_1,\dots, \kappa_\omega,\kappa_{\omega+1},\dots$ given by Proposition~\ref{prop:I2_implies_existence_of_omega_lined}. Maniplulating the associated embeddings $(j_i)$ as usual, we see  that $\kappa_0,\kappa_1,\kappa_2,...,\kappa_\omega$ witness that $\kappa_0,\kappa_1,\kappa_\omega$ are $\omega$-lined. 

    Recall that $p.init$ has a relaisation that was derived from Lemma~\ref{lem:dougherty_summary}. Replace the original embedding~$j$ used in that Lemma and the following development by $j_3$, since 
    \[
      j_3 : \kappa_0 \mapsto \kappa_1 \mapsto \kappa_\omega \mapsto \cdots.
    \]
    The resulting blp has the same first three cardinals, so we can now replace the linedness part, to get a bls $P$ realizing $p_{init}$ and satisfying $P.m = \omega$. By Lemma~\ref{lem:I2:omega_lined_implies_rank}, this finishes the proof.
\end{proof}

The author has calculated that $r(p_{\varepsilon_0})=\varepsilon_0$, and $r(p_{BO})=\psi(\Omega_{\omega})$, the well-known Buchholz ordinal. However the calculations are extremely long and tedious, and will be presented in another paper.

The author wrote a program to compute operations of blps in~\cite{blplink}. Note that the operation $.E(n)$ in this paper is equivalent with inputing 'E', then input the number $2^n$ in the program.

\subsubsection*{A very fast-growing function}

Another interesting application is a computable function, whose totality is only known under I2 now.

For a pair $(a,b)$ that is an element of $LTY_1$, if $b>0$, then there is $c>0$ such that $(a,b)=_{LTY_1}(a+1,c)$. The transformation $(a,b)\mapsto (a+1,c-1)$ makes a pair smaller in the sense of $LTY_1$.

\begin{definition}
  Given $n\in\omega$. We define the integer $F_1(n)$: start from the pair $(1,n)$, then repeatedly transform the pair until it reaches the form $(a,0)$, then let $F_1(n)=a$. 
\end{definition}

Assuming axiom I2, then $LTY_1$ is a well preorder, so the procedure must finally halt, and the function $F_1$ is total. Below is a python program that computes $F_1(n)$.

\begin{lstlisting}
def l(n,a,b):
  if a==2**n:return b
  if b==1:return a+1
  return l(n,l(n,a,b-1),a+1)
 
def F1(n):
  i=1;y=n
  while y>0:
    z=y
    while True:
      w=1
      for _ in range(i-1): w=l(z,1,w)
      if l(z,w,2**y)==2**z: z+=1
      else: break
    i+=1; y=z-2
  return i
 
print(F1(int(input())))

\end{lstlisting}

We expain the program. When $0\le a\le 2^z, 1\le b\le 2^z$, $l(n,a,b)=a\star_n b$ by definition of Laver tables, inside which we are working under $\mod 2^n$, that is, $2^n=0$. We start from $(i,y)=(1,n)$, and in each iteration we transform it to $(i+1,y'-2)$, where $y'$ is the least integer such that $(1\star_{y'}(1\star_{u'}...(1\star_{y'}1)))\star_{y'}2^y\ne 0$, where there are $i$ many $1$. The later number is nonzero only when $\gamma_{y'}>\text{crit}(j^{(i)}(\gamma_y))$ for any embedding $j$ and its critical points $\gamma_k$. So $(i,y)=_{LTY_1}(i+1,y'-1)$.

We have $F_1(0)=1, F_1(1)=4$, but $F_1(2)$ is very large. 

\medskip

\subsubsection*{Generalizing to $n>1$}

We can generalize the construction of $\vec{j}_1$ and $LTY_1$ further by the following way, note that for $n=1$ the construction below gives the $\vec{j}_1$ above.

Given an sequence of elementary embeddings $V_\lambda\rightarrow V_\lambda$, $(j_1,j_2,...)$, we denote $\mathrm{type}(j)=\{(i,i',e)|e(j_{i'})=j\}$, where $e$ is an expression of composition of applications and compositions, encoded as an integer.

\begin{definition}
    Assume $j:V_\lambda\rightarrow V_\lambda$ is $\Sigma^1_{2n}$. We will recursively choose increasing $\lambda$-sequences $\vec{j}_0,\vec{j}_1,...,\vec{j}_n$, ensuring for each $0\le k\le n$, $\vec{j}_{k}$ is computable, its witnessing expressions only depend on $k$ (not even $n$), and each element of $\vec{j}_{k}$ is $\Sigma^1_{2n-2k}$. Take $\vec{j_0}=(j,j,j,...)$.

    After choosing $\vec{j}_{k-1}$ for $0<k\le n$, we choose $\vec{j}_{k}$ in the following way. For each $i>0$, we define a sequence of embeddings $j_{i,i}, j_{i,i+1}, j_{i,i+2},\dots$. 
    We start with $j_{1,i}=(\vec{j}_{k-1})_i$ (the $i$th element of the previous sequence $\vec{j}_{k-1}$). Let $i>0$ and suppose that $j_{i,i},j_{i,i+1},j_{i,i+2},\dots$ were already defined. We choose $j_{i+1,i+1},j_{i+1,i+2},j_{i+1,i+3},\dots$ such that 
    \[
    \mathrm{type}(j_{i+1}(j_{i,i+1}),j_{i,i+1},j_{i+1}(j_{i,i+2}),j_{i,i+2},j_{i+1}(j_{i,i+3}),j_{i,i+3},\dots)
    \] 
    is equal to 
    \[
    \mathrm{type}(j_{i,i+1},j_{i+1,i+1},j_{i,i+2},j_{i+1,i+2},j_{i,i+3},j_{i+1,i+3},\dots).
    \] 
    This is possible because letting  
    \[ x=\mathrm{type}(j_{i+1}(j_{i,i+1}),j_{i,i+1},j_{i+1}(j_{i,i+2}),j_{i,i+2},j_{i+1}(j_{i,i+3}),j_{i,i+3},\dots)\in V_{\omega\cdot 2},
    \] 
    the proposition:
      \[
      \varphi(X):=\exists j'_1,j'_2,\dots(\mathrm{type}(X_1,j'_1,X_2,j'_2,\dots)=X_0, \text{ and each }j'_{i'}\text{ is }\Sigma^1_{2n-2k})
      \]
       is $\Sigma^1_{2n+2-2k}$, and $\varphi(j_{i+1}^+(A))$ is true, witnessed by $j'_{i'}=j_{i,i+i'}$ for each $i'>0$, where $A=\{(t,y,z)|1<t<\omega, j_{i,i+t}(y)=z\}\cup \{(0,x)\}$, and $X_t$ decodes this encoding. So we can take $j_{i+1,i+i'}$ as some witness $j'_{i'}$ of $\varphi(A)$. 

    Finally, $\vec{j}_k=(j_{1,1},j_{2,2},j_{3,3},\dots)$.

    From this construction, if $i\le i'$, then $j_{i,i'}\in j_{i,i'+1}$, and $j_{i,i'+1}\in j_{i+1,i'+1}$, and their witnessing expressions are both uniformly computable. So the recursive construction keeps going.

    We take $LTY_n=LTY_{\vec{j}_n}$. 
\end{definition}

\section{Future Work}
\label{sec:future_work}

The author used a program to verify many blps generated from $p_{init}$. All of them are copyable. However, the proof of this phenomenon is unknown.

\begin{question}
    Is it true that for any $0<n<\omega$, during the calculation of $f(p_{init},n)$ from definition, no non-copyable blp will appear?
\end{question}

If a non-copyable blp really appears, then it seems that simply deleting the last row would weaken our estimate. This is the situation when estimating $\kappa^9_3$.

Let $\theta_0,\dots,\theta_8$ be $\kappa_0,\kappa_1,\kappa_2,\kappa_3,\kappa^7_2,\kappa^{11}_{3},\kappa^{10}_{3},j_{(11)}(\kappa^{11}_3),\kappa^9_3$. From~\cite{Dougherty}, we know that $j_{(11)}\overset{\underset{\theta_4}{}}{\equiv} j\circ j$, and $j_{(10)}(\theta_4)=\theta_8$. After applying $j_{(10)}$ on it, we get $j_{(10)}(j_{(11)})\overset{\underset{\theta_8}{}}{\equiv} j_{(11)}\circ j_{(11)}$.

Also from~\cite{Dougherty}, $j_{(11)}(j_{(11)}((\theta_1,\theta_2,\theta_3)))=(\theta_5,\theta_6,\theta_7)$, and 
 \[
 j_{(11)}:(2)\theta_0\mapsto\theta_1\mapsto\theta_2\mapsto\theta_3\mapsto\theta_4\mapsto\theta_5\mapsto\theta_6\mapsto\theta_7,
 \]
 So
  \[
 j_{(10)}(j_{(11)}):(4)\theta_0\mapsto\theta_1\mapsto\theta_2\mapsto\theta_3\mapsto\theta_4\mapsto\theta_5\mapsto\theta_6\mapsto\theta_7\mapsto\theta_8.
 \]

Ignoring the condition $|s_i|\le 2\cdot l_i+2$, and by setting $j_1=j_2=j, j_3=j_4=j_5=j_6=j_{(11)}, j_7=j_{(10)}(j_{(11)})(j_{(11)})$, these embeddings witness a bls with the following blp as in Figure~\ref{fig:k93_blp}:

\drawSequence{0.1}{0.5}{0.5}{0,1,2,-1,0,1,2,3,-1,0,1,2,3,4,-2,0,1,2,3,4,5,-2,0,1,2,3,4,5,6,-2,0,1,2,3,4,5,6,7,-2,4,5,6,7,8,-2}[A blp ignoring certain conditions for estimating $\kappa^9_3$.][fig:k93_blp]

the 3rd, 4th, 5th, and 6th rows are represented by the same embedding. However, there is some trouble in going further. It seems that a new idea is required, to give a good estimate of $\kappa^9_3$, yet alone $F(4)$.

\smallskip

For the LTY ordinals, nearly nothing is known. We propose a question:

\begin{question} 
    Is it true that, for any increasing $\lambda$-sequence $\vec{j}$, if $(i,i')$ has height $\alpha$ in $LTY_{\vec{j}}$, then the order in the Remark~\ref{lty'} has height $2^\alpha$?
\end{question}

It's easy to see that this is true when $\alpha$ is finite, but for infinite $\alpha$ the answer is still unclear.

\smallskip

The height of $LTY_n$ and the rank of $p_{init}$ are also unclear, even for $n=1$.

\begin{question}
    Calculate the height of $LTY_n$ for each $n$. Calculate $r(p_{init})$.
\end{question}

\end{document}